\documentclass[reqno,a4paper,11pt]{article}
\usepackage{CJK,amsmath,amsthm,dsfont,amsfonts,amssymb,fancyhdr}
\usepackage{enumerate}
\usepackage{bbm,color,soul,supertabular,longtable,verbatim,extarrows}
\usepackage{titlesec}
\usepackage{graphicx}
\usepackage{cite}
\usepackage{graphicx,multirow,bm,rotating}
\usepackage{appendix}
\usepackage{booktabs,multirow,makecell}
\usepackage{tikz}
\usepackage{verbatim}
\usepackage[noblocks]{authblk}
\usepackage{ulem}
\usepackage{cancel}
\usepackage{subcaption}
\usepackage{booktabs,multirow,makecell}
\usepackage[top=2.4cm,bottom=2.2cm,left=2.6cm,right=2cm]{geometry}
\usepackage{hyperref}
\usepackage{subcaption}
\usepackage{amsthm}
\usepackage{amsmath}
\usepackage{amsmath}
\allowdisplaybreaks[1]  

\newtheorem{claim}{Claim}[section]
\newtheorem{conjecture}{Conjecture}

\newtheorem{theorem}{Theorem}
\theoremstyle{definition}

\theoremstyle{plain}
\newtheorem{lemma}[theorem]{Lemma}

\newtheoremstyle{noparens}%
  {}{}%
  {\itshape}{}%
  {\bfseries}{.}%
  { }%
  {\thmname{#1}\thmnumber{ #2}\mdseries\thmnote{ #3}}

\theoremstyle{noparens}
\newtheorem{lemmaNoParens}[theorem]{Lemma}
\newtheorem{theoremNoParens}[theorem]{Theorem}
\newtheorem{conjectureNoParens}[conjecture]{Conjecture}

\title{\bf The maximum number of maximal dissociation sets in trees\thanks{This work was supported by Beijing Natural Science Foundation (Grant No. 1252010).}}

\author{Meiqin Wang,  \quad Min Xu\thanks{\hangindent=1.5em Corresponding author. \newline {\em E-mail address:} xum@bnu.edu.cn (M. Xu).}, \quad Ning Zhang 

}

\date{}

\begin{document}
\maketitle
\vspace{-2.5em}
\begin{center}
    \emph{School of Mathematical Sciences, Beijing Normal University,} \\
    \emph{Key Laboratory of Mathematics and Complex Systems, Ministry of Education,} \\
    \emph{Beijing, 100875, China}
\end{center}

\renewcommand\abstractname{Abstract}
\begin{abstract}
Let $G$ be a simple graph. A dissociation set of $G$ proposed by Yannakakis in $1981$ is defined as a set of vertices that induces a subgraph in which every vertex has a degree of at most $1$. A dissociation set is maximal if it is not contained as a proper subset in any other dissociation set.
In $2025$, Wang et al.\cite{ZiyuanWang} established that for any tree $T$ of order $n\geq 4$, the number of maximal dissociation sets in $T$ is at most $3^{\frac{n-1}{3}}+\frac{n-1}{3}$ and characterized the extremal trees attaining the upper bound. They also proposed a conjecture about the upper bound of the maximal dissociation set. In this paper, we consider this conjecture and show that the maximum number of maximal dissociation sets in a tree of order $n(n\geq 3)$ is $g(n)$, where
\[
g(n) =
\begin{cases}
n,  &  n=3,4,5,6,\\
3^{\frac{n-1}{3}}+\frac{n-1}{3}, &  n \equiv 1 \pmod{3},~n\geq7,\\
4\cdot 3^{\frac{n-5}{3}}+n-5, &  n \equiv 2 \pmod{3},~n\geq8, \\
16\cdot 3^{\frac{n-9}{3}}+3n-25, &  n \equiv 0 \pmod{3},~n\geq12~\text{and }~n\neq21, \\
19, &  n=9, \\
1349, & n=21.
\end{cases}
\]
We also characterize the extremal trees with the maximum number of maximal dissociation sets.\\
\emph{Keywords}: Independent set; Dissociation sets; Tree.
\end{abstract}
\section{Introduction}
\quad\,\,\,Let $G = (V(G), E(G))$ be a graph. A set $S\subseteq V(G)$ is called an independent set of $G$ if no two vertices of S are adjacent in G. A maximal independent set is an independent set that is not a proper subset of any other independent set.

In the 1960s, Erd\H{o}s and Moser proposed the problem of determining the maximum number of maximal independent sets among all the graphs of order $n$. This problem was solved by Erd\H{o}s, and later Moon and Moser\cite{J.W.MoonL.Moser}. Regarding this problem, various families of graphs have been examined, such as trees, forests, connected graphs, bipartite graphs, unicyclic graphs, graphs with at most $r$ cycles,
and other graph families, have been examined \cite{K.M.KohC.Y.GohF.M.Dong, H.Law, J.Liu, B.E.SoganV.R.Vatter, S.G.Wagner, I.Wlch, J.Zito}.
The concept of the dissociation set was first introduced in 1981 by Yannakakis \cite{M.Yannakakis}; it refers to a vertex set $D$ of a graph $G$ such that the induced subgraph $G[D]$ has a maximum degree at most $1$. In particular, a dissociation set is maximal if it is not contained as a proper subset in any other dissociation set, and maximum if it has maximum cardinality. The dissociation number of G represents the cardinality of a maximum dissociation set of $G$. In the past forty years, dissociation sets have been investigated in \cite{F.ockJ.Pardey, F.BockJ.PardeyL.D.Penso, S.ChengB.Wu, Y.OrlovichA.Dolgui, W.SunS.Li, J.TuZ.ZhangY.Shi, M.Yannakakis}.

 We denote the set of all maximal dissociation sets of $G$ by $MD(G)$ and its cardinality by $\phi(G)=|MD(G)|$. 
 Recently, Cheng and Wu\cite{S.ChengB.Wu} considered this problem on forests and determined the largest number of maximal dissociation sets in forests of order $n$. They gave the following results.

\begin{theoremNoParens}[\cite{S.ChengB.Wu}]\label{S.ChengB.Wu}
For any forest $F$ of order $n\geq3$, then $\phi(F)\leq f(n),$ where  
 \[
f(n) :=
\begin{cases}
3^{\frac{n}{3}}, &  n \equiv 0 \pmod{3}~\text{and}~ n\geq 3,\\
4\cdot 3^{\frac{n-4}{3}}, &  n \equiv 1 \pmod{3}~\text{and}~n\geq 4, \\
5,                        &n=5,\\
16\cdot 3^{\frac{n-8}{3}}, &  n \equiv 2 \pmod{3}~\text{and}~n\geq 8.\\
\end{cases}
\]
with equality if and only if
\[
F \cong
\begin{cases}
\frac{n}{3}P_3, & \text{if } n \equiv 0 \pmod{3},\\
\frac{n-4}{3}P_3\cup K_{1,3}, & \text{if } n \equiv 1 \pmod{3},\\
K_{1,4}, & \text{if } n=5,\\
\frac{n-8}{3}P_3\cup 2K_{1,3}, & \text{if } n \equiv 2 \pmod{3}~\text{and}~ n\geq 8.
\end{cases}
\]
\end{theoremNoParens}

 Wang, Zhang, Tu and Xiong \cite{ZiyuanWang} determined the second largest number of maximal dissociation sets in forests of order $n$.  
They also considered this problem on trees and determined the largest number of maximal dissociation sets in trees of order $n$. They gave the following results.
\begin{theoremNoParens}[\cite{ZiyuanWang}]\label{Wang,Zhang,Tu,Xiong}
$(1)$ For any forest $F$ of order $ n\geq4$, if $\phi(F)\leq f(n),$ then
$\phi(F)\leq f_2(n),$ where 
\[
f_2(n) :=
\begin{cases}
n-1, & n=4,5\\
6, &n=6,\\
11\cdot3^{\frac{n-7}{3}}, &  n \equiv 1 \pmod{3}~\text{and}~n\geq 7,\\
15\cdot3^{\frac{n-8}{3}}, &  n \equiv 2 \pmod{3}~\text{and}~n\geq 8,\\
20, &n=9,\\
64\cdot3^{\frac{n-12}{3}}, &  n \equiv 0 \pmod{3}~\text{and}~n\geq 12.
\end{cases}
\]

$(2)$ For any tree $T$ of order $n\geq4$, then 
$$\phi(T)\leq 3^{\frac{n-1}{3}}+\frac{n-1}{3},$$ 
with equality if and only if $n \equiv 1 \pmod{3}$ and $T\cong {T_n}^*.$
\end{theoremNoParens}

\begin{figure}[htbp] 
 \centering
\begin{minipage}[b]{0.18\textwidth} 
        \includegraphics[width=\linewidth]{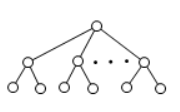} 
        \centering
 \caption{${T_n}^*$}
    \end{minipage}
\end{figure} 
To illustrate the upper bound more precisely, Wang, Zhang, Tu and Xiong  \cite{ZiyuanWang} proposed the following conjecture.
 \begin{conjectureNoParens}[\cite{ZiyuanWang}]
The largest number of maximal dissociation sets in trees of order $n\geq 7$ is
\[\begin{cases}
3^{\frac{n-1}{3}}+\frac{n-1}{3}, &  n \equiv 1 \pmod{3},\\
4\cdot 3^{\frac{n-5}{3}}+n-5, &  n \equiv 2 \pmod{3}, \\
16\cdot 3^{\frac{n-9}{3}}+3n-25, &  n \equiv 0 \pmod{3}. \\
\end{cases}\]
The extremal trees that achieve this largest number are depicted below.
\begin{figure}[htbp] 
 \centering
\begin{minipage}[b]{0.7\textwidth} 
        \includegraphics[width=\linewidth]{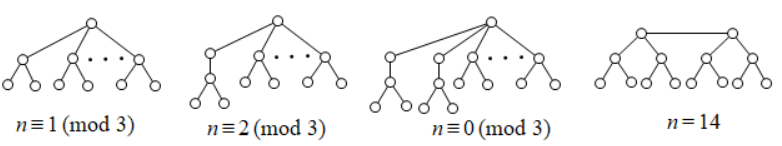} 
        \centering
    \end{minipage}
\end{figure} 
\end{conjectureNoParens}

We correct the above conjecture and prove the result below.
\begin{theorem}\label{main}
For any tree $T$ of order $n\geq3$, then $\phi(T)\leq g(n),$ where
\[
g(n) =
\begin{cases}
n,  &  n=3,4,5,6,\\
3^{\frac{n-1}{3}}+\frac{n-1}{3}, &  n \equiv 1 \pmod{3},~n\geq7,\\
4\cdot 3^{\frac{n-5}{3}}+n-5, &  n \equiv 2 \pmod{3},~n\geq8, \\
16\cdot 3^{\frac{n-9}{3}}+3n-25, &  n \equiv 0 \pmod{3},~n\geq12~\text{and }~n\neq21, \\
19, &  n=9, \\
1349, & n=21.
\end{cases}
\]
The extremal trees that achieve this largest number are illustrated in Figure \ref{fig2}. 

\begin{figure}[htbp] 
 \centering
\begin{minipage}[b]{0.92\textwidth} 
        \includegraphics[width=\linewidth]{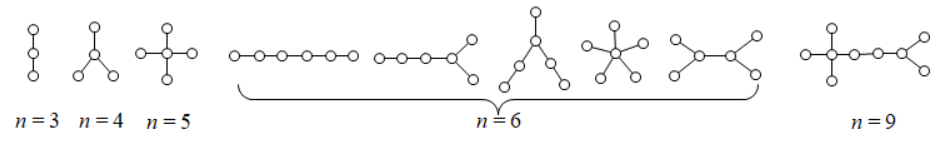} 
        \centering
    \end{minipage} \\
\begin{minipage}[b]{0.92\textwidth} 
        \includegraphics[width=\linewidth]{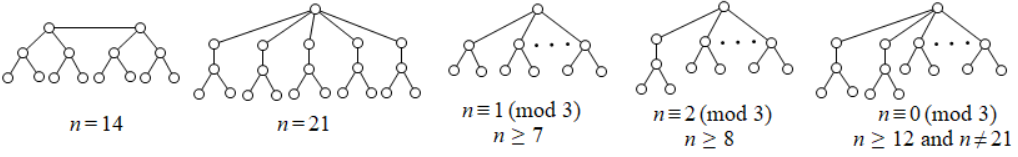} 
        \centering
    \end{minipage} 
    \caption{The extremal trees}
    \label{fig2}
\end{figure}
\end{theorem}

\section{Preliminary lemmas}
\quad\,\,\,Let $G=(V,E)$ be a simple graph. For $v\in V(G)$, let $N_G(v)=\{u\in V:uv\in E\}$ and $N_G[v]=\{v\}\cup N_G(v)$. We define the degree of $v$ as $d_G(v) = |N_G(v)|$. A vertex $v$ is called a leaf if $d(v) = 1$ and a vertex is called a support vertex if it is adjacent to a leaf. 
We follow \cite{J.A.Bondy} for graph-theoretical terminology and notation not defined here. 

Let $MD(G)$ be the set of all maximal dissociation sets of $G.$ For $v\in V(G)$, we introduce the
following notations, which give a partition of $MD(G)$:
\begin{itemize}
    \item $MD(G,\bar{v})=\{S:S\in MD(G), v\notin S\}$ and $MD(G,v)=\{S:S\in MD(G), v\in S\}$,
    \item $MD(G, v^{i})=\{S:S\in MD(G), v\in S, d_{G[S]}(v)=i\}$ for $i\in\{0,1\}$.
\end{itemize}
Clearly, $MD(G, v)=MD(G, v^0)\cup MD(G, v^1)$, where $G[S]$ is the subgraph of $G$ induced by $S$.

Let $\phi(G,\bar{v})$, $\phi(G,v)$ and $\phi(G,v^i)$ be the cardinalities of $MD(G,\bar{v})$,
$MD(G,v)$ and $MD(G,v^i)$, respectively. Obviously,
$\phi(G) = \phi(G,\bar{v})+\phi(G,v)=\phi(G,\bar{v})+\phi(G,v^0)+\phi(G,v^1)$. Furthermore, for $\{x_1,..., x_r\}\subset V(G)$, $\hat{x}_i\in\{\bar{x}_i,x_i,{x_i}^0,{x_i}^1\},$ let $MD(G,\hat{x}_1\hat{x}_2\ldots\hat{x}_r)=\bigcap\limits_{i=1}^r MD(G,\hat{x}_i)$
and $\phi(G,\hat{x}_1\hat{x}_2\ldots\hat{x}_r)=|MD(G,\hat{x}_1\hat{x}_2\ldots\hat{x}_r)|$.
\begin{lemma}\label{1}
With the notations above, the following statements hold.
\item[\rm{(i)}] $f(m_1)\cdot f_2(m_2)\leq f_2(m_1+m_2)$ for $m_1\geq3,m_2\geq4.$
\item[\rm{(ii)}] $ \frac{f(n)}{f(n+1)}\leq \frac{4}{5}$ for $n\geq 3$; $ \frac{f_2(n)}{f_2(n+1)}\leq \frac{3}{4}$ for $n\geq 4$; $\frac{g(n)}{g(n+1)}\leq \frac{3}{4}$ for $n=7$ and $n\geq 9$, $\frac{g(8)}{g(9)}=\frac{15}{19}$;
$\frac{g(n)}{f(n)}\leq \frac{11}{12}$ for $n\geq 6$ and $n\neq 8$, $\frac{g(8)}{f(8)}=\frac{15}{16}$; $g(n)\geq16\cdot 3^{\frac{n-9}{3}}+3n-25$ for $n\geq 14$. 
\end{lemma}
\begin{proof}
$(i)$ Let $F_1$ and $F_2$ be forests of orders $m_1$ and $m_2$, respectively. If forest $F_2$ does not attain the upper bound $f(m_2)$, then Lemma \ref{S.ChengB.Wu} implies that forest $F_1\cup F_2$ does not attain the upper bound $f(m_1+m_2)$ either. Consequently, we have $f(m_1)\cdot f_2(m_2)\leq f_2(m_1+m_2)$.

$(ii)$ By direct verification, the result follows when $n\leq9$. For $n\geq10$, we have
\begin{center}
\begin{minipage}{0.45\linewidth}
\[
\frac{f(n)}{f(n+1)}=
\begin{cases}
\displaystyle\frac{3}{4}, &  n \equiv 0 \pmod{3}, \\
\displaystyle\frac{3}{4}, &  n \equiv 1 \pmod{3},\\
\displaystyle\frac{16}{27},&  n \equiv 2 \pmod{3}.\\
\end{cases}
<\frac{4}{5},
\]
\end{minipage}
\hspace{0.5cm}
\begin{minipage}{0.45\linewidth}
\[
\frac{f_2(n)}{f_2(n+1)}=
\begin{cases}
\displaystyle\frac{64}{99}, &  n \equiv 0 \pmod{3}, \\
\displaystyle\frac{11}{15}, &  n \equiv 1 \pmod{3},\\
\displaystyle\frac{45}{64},&  n \equiv 2 \pmod{3}.\\
\end{cases}
<\frac{3}{4},
\]
\end{minipage}
\end{center}

\begin{align*}
\frac{g(n)}{g(n+1)}=
\begin{cases}
\displaystyle\frac{16}{27}+\frac{227n-2025}{ 3^{\frac{n+12}{3}}+27n}, &  n \equiv 0 \pmod{3},~n\geq12,~n\neq21,\\
\displaystyle\frac{1349}{2194}, & n=21,\\
\displaystyle\frac{3}{4}+\frac{32-5n}{4\cdot 3^{\frac{n-4}{3}}+n-4}, &  n \equiv 1 \pmod{3},\\
\displaystyle\frac{3}{4}+\frac{184-20n}{16\cdot 3^{\frac{n-8}{3}}+3n-22},&  n \equiv 2 \pmod{3},~n\geq11,~n\neq20,\\
\displaystyle\frac{987}{1349}, & n=20.\\
\end{cases}
<\frac{3}{4},
\end{align*}

\begin{align*}
\frac{g(n)}{f(n)}=
\begin{cases}
\displaystyle\frac{16}{27}+\frac{3n-25}{ 3^{\frac{n}{3}}}, &  n \equiv 0 \pmod{3},~n\geq12,~n\neq21,\\
\displaystyle\frac{1349}{2187}, & n=21,\\
\displaystyle\frac{3}{4}+\frac{n-1}{12\cdot 3^{\frac{n-4}{3}}}, &  n \equiv 1 \pmod{3},\\
\displaystyle\frac{3}{4}+\frac{n-5}{16\cdot 3^{\frac{n-8}{3}}},&  n \equiv 2 \pmod{3}.\\
\end{cases}
<\frac{11}{12}.
\end{align*}

We now provide $g(n)\geq16\cdot 3^{\frac{n-9}{3}}+3n-25$ for $n\geq 14$.

For $n \equiv 0 \pmod{3}$ and $n\geq14$, we have 
\begin{align*}
g(n)&=
\begin{cases}
16\cdot 3^{\frac{n-9}{3}}+3n-25, &  n \equiv 0 \pmod{3},~n\neq21, \\
1349, & n=21.\\
\end{cases}\\
&\geq 16\cdot 3^{\frac{n-9}{3}}+3n-25.
\end{align*}

For $n \equiv 1 \pmod{3}$ and $n \equiv 2 \pmod{3}$, we proceed by induction to show that $g(n)\geq16\cdot 3^{\frac{n-9}{3}}+3n-25$. We first verify the conclusion directly for  $n = 14$  and  $n = 16$.
We assume that $g(n)\geq16\cdot 3^{\frac{n-9}{3}}+3n-25$ and now prove that $g(n+3)\geq16\cdot 3^{\frac{n-6}{3}}+3n-16$. Note that
\begin{align*}
&g(n+3)-16\cdot 3^{\frac{n-6}{3}}-3n+16\\
=&
\begin{cases}
3^{\frac{n+2}{3}}+\frac{n+2}{3}-16\cdot 3^{\frac{n-6}{3}}-3n+16, &  n \equiv 1 \pmod{3}, \\
4\cdot 3^{\frac{n-2}{3}}+n-2-16\cdot 3^{\frac{n-6}{3}}-3n+16, &  n \equiv 2 \pmod{3}.\\
\end{cases}\\
=&
\begin{cases}
[3^{\frac{n-1}{3}}+\frac{n-1}{3}-16\cdot 3^{\frac{n-9}{3}}-3n+25]+(2\cdot3^{\frac{8}{3}}-32)\cdot 3^{\frac{n-9}{3}}-8, &  n \equiv 1 \pmod{3}, \\
[4\cdot 3^{\frac{n-5}{3}}+n-5-16\cdot 3^{\frac{n-9}{3}}-3n+25]+(8\cdot3^{\frac{4}{3}}-32)\cdot 3^{\frac{n-9}{3}}-6, &  n \equiv 2 \pmod{3}.\\
\end{cases}\\
>&0.
\end{align*}
The second inequality holds because $g(n)\geq16\cdot 3^{\frac{n-9}{3}}+3n-25$, $(2\cdot3^{\frac{8}{3}}-32)\cdot 3^{\frac{n-9}{3}}-8>0$ for $n\geq15$ and $(8\cdot3^{\frac{4}{3}}-32)\cdot 3^{\frac{n-9}{3}}-6>0$ for $n\geq14$.
\end{proof}

This method is applied consistently in the remainder of the paper, thereby simplifying the comparison of inequalities.

\begin{lemmaNoParens}[\cite{ZiyuanWang}\label{tu1}]
Let $T$ be a tree. If there exists a support vertex $v$ of degree $2$ in $T$ that is adjacent to a leaf $u$ and a nonleaf vertex $x$, then $\phi(T)\leq\phi(T')$,
where $T'$ is the tree obtained from $T$ by deleting the edge $uv$ and adding a new edge $ux$, i.e., $T'=T-uv+ux$. Moreover, if $N_T(x)=\{v,x_1,\ldots,x_l\}$ and there exists a component in $T-\{u,v,x, x_1,\ldots,x_l\}$ of order at least $3$, then $\phi(T)<\phi(T')$.
\end{lemmaNoParens}

\begin{lemmaNoParens}[\cite{ZiyuanWang}\label{tu2}]
Let $T$ be a tree and $k\geq2$ be an integer. Let $v$ be a support vertex in $T$ of degree $k+1$ that is adjacent to $k$ leaves $u_1,\ldots,u_k$ and a nonleaf vertex $t$. If the vertex $t$ is adjacent to a leaf $x$, then $\phi(T)\leq\phi(T')$, where $T'=T-vt+vx$. Moreover, if $|V(T)\setminus\{u_1,\ldots,u_k,v,x\}|\geq3$, then $\phi(T)<\phi(T')$.
\end{lemmaNoParens}

\begin{lemmaNoParens}\label{tu0}
Let $\widetilde{T}$ be a tree satisfying $\phi(\widetilde{T})=\max\{\phi({T}) :{T} \text{~is a tree and }|V({T})|=n\}.$ Then, $\phi(\widetilde{T})\geq g(n)$.
\end{lemmaNoParens}
\begin{proof}
To prove the result, we will show that $\phi({T})=g(n)$ where $T$ are the trees in Fig.~\ref{fig2}. We choose $n=3,4,9,21$ for illustration. The other case can be proven by the same discussion.

If $n=3$, as shown in Fig.~\ref{L7-1}~$(a)$, the graph $T$ has exactly three maximal dissociation sets $\{v,v_1\}$, $\{v,v_2\}$, and $\{v_1,v_2\}$. Thus, $\phi({T})=3=g(3)$. 

If $n=4$, as shown in Fig.~\ref{L7-1}~$(b)$, the graph $T$ has exactly four maximal dissociation sets $\{v,v_1\}$, $\{v,v_2\}$, $\{v,v_3\}$ and $\{v_1,v_2,v_3\}$. Thus, $\phi({T})=4=g(4)$. 

If $n=9$, as shown in Fig.~\ref{L7-1}~$(c)$, we have $\phi({T})=\phi({T},\bar{v})+\phi({T},v^0)+\phi({T},v^1)=\phi({T},\bar{v}v_1v_2v_3)+\phi({T},v^0)+\sum\limits_{i=1}^{3}\phi({T},vv_i)+\phi({T},vv_4)=4+0+3\cdot4+3=19=g(19)$.

If $n=21$, as shown in Fig.~\ref{L7-1}~$(d)$, we have\\
$\phi({T})=\phi({T},\bar{v})+\phi({T},v^0)+\phi({T},v^1)\\
~~~\,~~\,~\,=[\phi({T}-{v})-\phi({T}-{v},\bar{v}_1\cdots\bar{v}_5)-\sum\limits_{i=1}^{5}\phi({T}-{v},\bar{v}_1\cdots\bar{v}_{i-1}{{v}_i}^0\bar{v}_{i+1}\cdots\bar{v}_{5})]+\phi({T},v^0)+\sum\limits_{i=1}^{5}\phi({T},vv_i)\\
~~~\,~~\,~=(4^5-2^5-5\cdot2^4)+2^5+5\cdot3^4\\
~~~\,~~\,~=1349=g(21).$

Overall, we have $\max\{\phi({T}) :{T} \text{~is a tree and }|V({T})|=n\}\geq g(n),$ and the lemma follows.
\end{proof}
\begin{figure}[htbp]
\centering
\begin{minipage}[t]{0.5\textwidth}
\centering
\includegraphics[width=\linewidth]{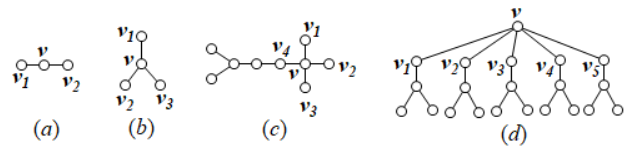}
\caption{}
\label{L7-1}
\end{minipage}
\hspace{0.2cm}
\begin{minipage}[t]{0.13\textwidth}
\centering
\includegraphics[width=\linewidth]{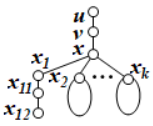}
\caption{}
\label{127}
\end{minipage}
\hspace{0.2cm}
 \begin{minipage}[t]{0.15\textwidth}
        \centering
        \includegraphics[width=\linewidth]{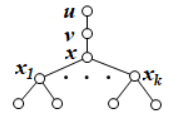}
        \caption{}
        \label{128}
    \end{minipage}
\end{figure}
\begin{lemma}\label{tu3}
Let $\widetilde{T}$ be a tree of order $n\geq10$ and $\phi(\widetilde{T})=\max\{\phi({T}) : {T} \text{~is a tree and }|V({T})|=n\}$. Assume that $v\in V(\widetilde{T})$ and that the components of $\widetilde{T}-v$ include $p$ copies of $K_1$ and $q$ copies of $K_2$. If $\phi({T})\leq g(n-1)$ for every tree $T$ of order $n-1$, then $0\leq p\leq2$, $ q=0.$ 
\end{lemma}
\begin{proof}
Let $\widetilde{T}$ be a tree of order $n\geq10$ and $\phi(\widetilde{T})=\max\{\phi({T}) : {T} \text{~is a tree and }|V({T})|=n\}$. Thus, $\widetilde{T}\geq g(n)$ by Lemma \ref{tu0}. Assume that $v\in V(\widetilde{T})$ and that the components of $\widetilde{T}-v$ include $p$ copies of $K_1$ and $q$ copies of $K_2$. Let $N_{\widetilde{T}}(v)=\{u_1,\ldots,u_p,x_1,\ldots,x_l\}$ where $u_1,\ldots,u_p$ are leaves and $x_1,\ldots,x_l$ are nonleaf  vertices. If $l=0$, then $\widetilde{T}$ is $K_{1,p}$. Thus, by the same discussion of Lemma \ref{tu0}, we have $\phi(\widetilde{T})=n<g(n)$, which contradicts the choice of $\widetilde{T}$. Therefore, $l\geq1$. If $p\geq3$, then $\phi(\widetilde{T},\bar{v})=\phi(T',\bar{v}).$  Let $T'=\widetilde{T}-u_p$; then, $T'$ is a tree. Thus,
\begin{align*}
\phi(\widetilde{T})&=\phi(\widetilde{T},\bar{v})+\phi(\widetilde{T},v^0)+\phi(\widetilde{T},v^1)\\
&=\phi(\widetilde{T},\bar{v})+\phi(\widetilde{T},v^0)+[\sum\limits_{i=1}^{l}\phi(\widetilde{T},vx_i)+\sum\limits_{i=1}^{p-1}\phi(\widetilde{T},vu_i)+\phi(\widetilde{T},vu_p)]\\
&=\phi(T',\bar{v})+\phi(T',v^0)+[\sum\limits_{i=1}^{l}\phi(T',vx_i)+\sum\limits_{i=1}^{p-1}\phi(T',vu_i)]+\phi(\widetilde{T},vu_p)\\
&=\phi(T',\bar{v})+\phi(T',v^0)+\phi(T',v^1)+\phi(\widetilde{T},vu_p)\\
&=\phi(T')+\phi(\widetilde{T},vu_p).
\end{align*}

Note that with $|V(T')|=|V(T)|-1=n-1,$ we have $\phi(T')\leq g(n-1)$ by assumption. By Theorem \ref{S.ChengB.Wu}, $\phi(\widetilde{T},vu_p)=\phi(\widetilde{T},vu_p\bar{u}_1\cdots\bar{u}_{p-1}\bar{x}_1\cdots\bar{x}_{l})\leq f(n-5)$ for $p\geq3$ and $l\geq1$. Hence, we have $\phi(\widetilde{T})\leq g(n-1)+f(n-5)<g(n)$ for $n\geq10$. By Lemma \ref{tu0}, it is a contradiction.

 If $q\geq2$, by Lemma \ref{tu1}, there exists a tree $T'$ such that $\phi(T')\geq\phi(\widetilde{T})$, and $u$ has at least four leaf-neighbors in $T'$. Hence,  $\phi(T')\geq\phi(\widetilde{T})=\max\{\phi({T}) : {T} \text{~is a tree and }|V({T})|=n\}.$ According to the above discussion about $p\geq3$, this is a contradiction. Thus, $0\leq p\leq 2$.

 If $q=1$, it first follows that $p=0$. Otherwise, similar to case $q\geq2$, by Lemma \ref{tu1}, there exists a tree $T'$ such that $\phi(T')\geq\phi(\widetilde{T})$, and $u$ has at least three leaf-neighbors in $T'$. According to the above discussion about $p\geq3$, this is a contradiction. Since $q=1$, there exists a support vertex $v$ of degree $2$. Let $u$ be a leaf adjacent to $v$, and let $x$ be a nonleaf neighbor of $v$. Let $N(x)=\{v,x_1,x_2,\ldots,x_k\}$ and $X_i$ be the component of $\widetilde{T}-x$ containing the vertex $x_i$$(i\in[k])$.

If there exists a component in $\widetilde{T}-\{u,v,x,x_1,\ldots,x_k\}$ of order at least $3$, then by Lemma \ref{tu1}, we have $\phi(\widetilde{T})<\phi(\widetilde{T}-uv+ux)$, which contradicts $\phi(\widetilde{T})=\max\{\phi({T}) : {T} \text{~is a tree and }|V({T})|=n\}$. Each component of $\widetilde{T}-\{u,v,x,x_1,\ldots,x_k\}$ has order at most $2$. From the above discussion, for $w\in V(\widetilde{T})$, which is adjacent to $p$ leaves and $q$ edges, where $0\leq p\leq2$, $0\leq q\leq1$ and $\min\{p,q\}=0$, we have that $X_i-x_i$ is either two copies of $K_1$ or one copy of $K_2$.
If $X_1-x_1$ is a copy of $K_2=x_{11}x_{12}$ (see Fig.~\ref{127}), then $k=1$. Otherwise, $X_2$ is a component with three vertices in $T-\{x,x_1,x_{11},x_{12}\}$. By Lemma \ref{tu1}, we have $\phi(\widetilde{T})<\phi(\widetilde{T}-x_{11}x_{12}+x_{1}x_{12})$, which is a contradiction. Thus, $k=1$ and $n=6$, which contradicts $n\geq10$.
Therefore, every $X_i$ is isomorphic to $K_{12}$ for $i=1,2,\ldots,k$; see Fig.~\ref{128}. We have
\begin{align*}
\phi(\widetilde{T})&=\phi(\widetilde{T},\bar{x})+\phi(\widetilde{T},x^0)+\phi(\widetilde{T},x^1)\\
&=\phi(\widetilde{T},\bar{x})+\phi(\widetilde{T},u\bar{v}x^0\bar{x}_1\cdots\bar{x}_k)+[\phi(\widetilde{T},x^1v\bar{x}_1\cdots\bar{x}_k)+\sum\limits_{i=1}^{k}\phi(T',x\bar{v}\bar{x}_1\cdots\bar{x}_{i-1}x_i\bar{x}_{i+1}\cdots\bar{x}_{k})]\\
&=3^{k}+1+(1+k)\\
&=3^{\frac{n-3}{3}}+\frac{n+3}{3}<g(n),
\end{align*}
where the last equality holds for $k=\frac{n-3}{3}$ and, in this situation, contradicts the choice of $\widetilde{T}$. This completes the proof.
\end{proof}
Let $X$ and $Y$ be two isomorphic graphs, and let vertices $x\in V(X)$ and $y\in V(Y)$ be fixed. We call $(X,x)$ and $(Y,y)$ equivalent, denoted by $(X,x)\sim (Y,y)$, if there exists a graph isomorphism $\varphi\colon X\to Y$ such that $\varphi(x)=y$.
\begin{lemma}\label{zw4}
Let $T$ be a tree of order $n$ $(n\geq 10)$ such that every support vertex has a degree of at least $3$ and at most two leaf neighbors and that every support vertex of degree $3$ with exactly two leaf neighbors has a nonleaf neighbor of degree $2$. Assume that $\phi(S)\leq g(|S|)$ for every tree $S$ of order at most $n$.
If there exists an edge $xy\in E(T)$ such that both components of $T- xy$ have order at least $5$, then 
$\phi(X,x)\leq f_2(|V(X)|-1)+1,$ where $X$ denotes the component of $T-xy$ that contains vertex $x$. If the equality holds, then $|V(X)|=6,7,10$.
\end{lemma}
\begin{proof}
Let $xy\in E(T)$ such that both components of $T- xy$ have an order of at least $5$, and let $X$ denote the component of $T-xy$ containing vertex $x$. Let $|V(X)|= m.$ Note that $m\geq 5$.
\begin{figure}[htbp]
\centering
\begin{minipage}[t]{0.5\textwidth}
\centering
\includegraphics[width=\linewidth]{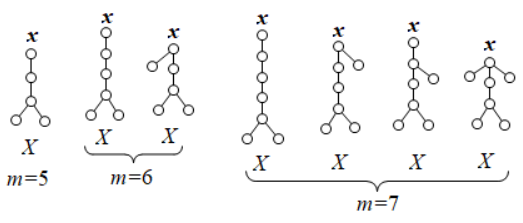}
\caption{}
\label{119}
\end{minipage}
\end{figure}

 When $5\leq m\leq 7$, according to the choice of ${T}$, the $(X,x)$ is as shown in Fig.~\ref{119}. It is easy to prove that the result holds.
Suppose that the result holds for $8\leq |V(X)|\leq m-1$. We now prove this for $|V(X)|= m\leq n-5.$

Let $N_X(x) = \{x_1, \dots, x_k\}$, with $x_1, \ldots, x_s$ being the leaf neighbors of $x$ in $X$. Then, $s \in \{0,1,2\}$. For each $i \in [k]$, let $X_i$ be the component of $X - x$ that contains $x_i$.
The structure of $T$ implies that for each $i$, either $X_i - x_i \cong P_3$ or it does not contain $P_3$ as a component, and moreover, $|X_i - x_i| = 0$ or $|X_i - x_i| \geq 3$.

\textbf{Case 1.} $s\neq0$.

In this case, $x$ is a support vertex; thus, $|N_T(x)| \geq 3$ and $|N_X(x)| \geq 2$, implying that $k \geq 2$.
Note that
\begin{align}
\phi(X,x)&=\phi(X,x^0)+\phi(X,x^1)=0+\sum\limits_{i=1}^{k}\phi(X,x\bar{x}_1\cdots\bar{x}_{i-1}x_{i}\bar{x}_{i+1}\cdots\bar{x}_{k})\\
&\leq\sum\limits_{i=1}^{k}f(m-k-|N(x_i)|)\notag\\
&\leq sf(m-k-1)+(k-s)f(m-k-2).
\end{align}

If $k = 2$, we have $s = 1$ for $|V(X)|\geq8$. From $(2)$, it follows that $\phi(X,x)\leq f(m-3)+f(m-4)\leq f_2(m-1)+1.$
Specifically, if the equality holds, then $m=10$.

If $k=3$, we first consider $s=1$. From $(2)$, we have $\phi(X,x)\leq f(m-4)+2f(m-5)\leq f_2(m-1)+1.$ Next, we consider that $s = 2$. If $\phi(X,x{x}_1\bar{x}_{2}\bar{x}_{3}) = \phi(X,x\bar{x}_1{x}_2\bar{x}_3) = f(m-4)$, then Theorem~\ref{S.ChengB.Wu} and the structure of $T$ imply $X-\{x,x_1,x_2,x_3\} \cong K_{1,3}$ or $2K_{1,3}$, meaning $(X,x) \sim (T_1,x)$ or $(T_2,x)$ (Fig.~\ref{311}).  Direct verification shows that $\phi(T_1,x) = 11 < 13=f_2(7)+1$ and $\phi(T_2,x) = 41 < 46=f_2(11)+1$. Otherwise, $\phi(X,x{x}_1\bar{x}_{2}\bar{x}_{3}) = \phi(X,x\bar{x}_1{x}_2\bar{x}_3) \leq f_2(m-4).$ Thus, from $(2)$, we have $\phi(X,x) \leq 2f_2(m-4) + f(m-5) < f_2(m-1) + 1$.

If $k \geq 4$, from $(2)$, we have
$\phi(X,x)\leq sf(m-k-1)+(k-s)f(m-k-2)\leq 2f(m-k-1)+(k-2)f(m-k-2)\overset{\triangle}{=}F(k).$ Combining $k \geq 4$ with the structure of $T$, we obtain $m-k-3\geq3$. Then, by Lemma~\ref{1}$(ii)$, we obtain 
{\small
\begin{align*}
F(k) -F(k+1)&=2f(m-k-1)+(k-2)f(m-k-2)-2f(m-k-2)-(k-1)f(m-k-3)\\
&\geq2\cdot\frac{5}{4}f(m-k-2)+(k-2)f(m-k-2)-2f(m-k-2)-(k-1)\cdot\frac{4}{5}f(m-k-2)\\
&= (\frac{k}{5}-\frac{7}{10})f(m-k-2)>0.
\end{align*}
}
 Hence, $\phi(X,x) \leq F(k) \leq F(4) = 2f(m-5) + 2f(m-6) < f_2(m-1) + 1$.

\begin{figure}[htbp]
\centering
\begin{minipage}[t]{0.15\textwidth}
        \centering
        \includegraphics[width=\linewidth]{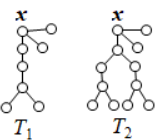}
        \caption{}
        \label{311}
    \end{minipage}
 \hspace{0.3cm}
     \begin{minipage}[t]{0.15\textwidth}
        \centering
        \includegraphics[width=\linewidth]{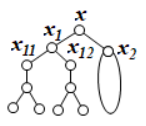}
        \caption{}
        \label{12}
    \end{minipage}
\hspace{0.3cm}
    \begin{minipage}[t]{0.12 \textwidth}
        \centering
        \includegraphics[width=\linewidth]{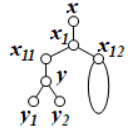}
        \caption{}
        \label{15}
    \end{minipage}
\hspace{0.2cm}
    \begin{minipage}[t]{0.16 \textwidth}
        \centering
        \includegraphics[width=\linewidth]{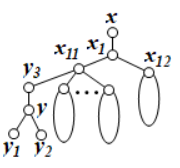}
        \caption{}
        \label{120}
    \end{minipage}
\end{figure}

\textbf{Case 2.} $s=0$.

We prove this case according to the degree of $x_i$ $(i\in[k]).$

\textbf{Subcase 2.1.} $d(x_i)\geq 3$ for all $i\in [k]$.

In this subcase, $X_i - x_i$ does not contain $P_3$ as a component.

\textbf{Subcase 2.1.1.} $k\geq 2$.

Note that $\phi(X,x)=\phi(X,x^0)+\phi(X,x^1)=\phi(X,x^0\bar{x}_1\cdots\bar{x}_{k})+\sum\limits_{i=1}^{k}\phi(X,x\bar{x}_1\cdots\bar{x}_{i-1}x_{i}\bar{x}_{i+1}\cdots\bar{x}_{k}).$

If for some $i \in [k]$, $\phi(X, x\bar{x}_1\cdots\bar{x}_{i-1}x_{i}\bar{x}_{i+1}\cdots\bar{x}_{k}) = f(m - k - |N(x_i)|),$ then by Theorem~\ref{S.ChengB.Wu} and the structure of $T$, it follows that $k = 2$ and $X_j - x_j \cong 2K_{1,3}$ for $j \in [2] \setminus \{i\}$. Consequently, the structure of $(X, x)$ is as shown in Fig.~\ref{12}. Then 
{\small
\begin{align*}
\phi(X,x)&=\phi(X,x^0)+\phi(X,x^1)\\
&\leq[\phi(X_1-x_1)-\phi(X_1-x_1,\bar{x}_{11}\bar{x}_{12})]\phi(X_2-x_2)+[\phi(X,x{x}_{1})+\phi(X,x{x}_{2})]\\
&\leq(16-4) f(m-11)+[ 9f(m-11)+16 f(m-12)]\\
&< f_2(m-1)+1.
\end{align*}
}
Otherwise, $\phi(X,x\bar{x}_1\cdots\bar{x}_{i-1}x_{i}\bar{x}_{i+1}\cdots\bar{x}_{k})\leq f_2(m - k - |N(x_i)|)$ for $i\in[k].$ From Theorems \ref{S.ChengB.Wu}, \ref{Wang,Zhang,Tu,Xiong} and the structure of $T$, we know that the forest $X - \{x_1,\ldots,x_{k}\}$ does not attain the upper bound $f(m - k - 1)$; hence, $\phi(X,x^0\bar{x}_1\cdots\bar{x}_{k}) \leq f_2(m-k-1)$. 
Therefore,
\begin{align}
\phi(X,x)&\leq\phi(X,x^0\bar{x}_1\cdots\bar{x}_{k})+\sum\limits_{i=1}^{k}\phi(X,x\bar{x}_1\cdots\bar{x}_{i-1}x_{i}\bar{x}_{i+1}\cdots\bar{x}_{k})\notag\\
&\leq f_2(m-k-1)+\sum\limits_{i=1}^{k}f_2(m-k-|N(x_i)|)\notag\\
&\leq f_2(m-k-1)+kf_2(m-k-3)\overset{\triangle}{=}F_2(k).
\end{align}
Combining $k \geq 2$ and $d(x_i)\geq3$ $(i\in[k])$ with the structure of $T$, we obtain $m-k-4\geq4$. Then, by Lemma~\ref{1}$(ii)$, we obtain 
{\small\begin{align*}
F_2(k) -F_2(k+1)&=f_2(m-k-1)+kf_2(m-k-3)-f_2(m-k-2)-(k+1)f_2(m-k-4)\\
&\geq(\frac{4}{3}-1)f_2(m-k-2)+[k-\frac{3}{4}(k+1)]f_2(m-k-3)\\
&\geq(\frac{1}{4}k-\frac{11}{36})f_2(m-k-3)>0.
\end{align*}}
 Thus, from $(3)$, we have
$\phi(X,x) \leq F_2(k) \leq F_2(2) = f_2(m-3) + 2f_2(m-5) < f_2(m-1) + 1.$

\textbf{Subcase 2.1.2.} $k= 1$.

Let $N_X(x_1) = \{x,x_{11}, x_{12}, \dots, x_{1l}\}$, and let $X_{1i}$ denote the component of $X - x_1$ containing $x_{1i}$ for each $i \in [l]$. If $l \geq 3$, by the structural conditions of $T$ and Theorems \ref{S.ChengB.Wu} and~\ref{Wang,Zhang,Tu,Xiong}, we have $\phi(X,x) = \phi(X,x^0) + \phi(X,x^1) \leq f_2(m-2) + f(m-5) < f_2(m-1) + 1.$

Now, consider the case $l = 2.$
If $|X_{11} - x_{11}| = 0$ (or $|X_{12} - x_{12}| = 0$), then $X_{11} = \{x_{11}\}$.
By Theorem~\ref{S.ChengB.Wu}, we have $\phi(X,x) = \phi(X,x^0) + \phi(X,x^1) \leq f(m-3) + f(m-4) \leq f_2(m-1) + 1,$ where equality holds only when $m = 10.$

Next, we consider that $|X_{11}-x_{11}|\neq0$ and $|X_{12}-x_{12}|\neq0$. The structure of $T$ implies that  $X_{11}-x_{11}$ (or $X_{12}-x_{12}$) does not contain any component of order $2$, but must contain at least one component of order at least $3$. If there exists a component with vertices $y,y_1,y_2$ in $X_{11}-x_{11}$ (or $X_{12}-x_{12}$), then $(X,x)$ is as shown in Fig.~\ref{15}. Note that $T-\{y,y_1,y_2\}$ and $T-\{y,y_1,y_2,x_{11}\}$ satisfy the conditions of Lemma \ref{zw4}; thus, by the induction hypothesis, we have
\begin{align*}
\phi(X,x)&=\phi(X,x\bar{y})+\phi(X,xy^0)+[\phi(X,xy^1y_1)+\phi(X,xy^1y_2)+\phi(X,xy^1x_{11})]\\
&=\phi(X-\{y,y_1,y_2\},x)+0+[2\cdot\phi(X-\{y,y_1,y_2,x_{11}\},x)+\phi(X,xy^1x_{11}\bar{y}_1\bar{y}_2\bar{x}_1)]\\
&\leq
\begin{cases}
[f_2(6)+1]+\{2[f_2(5)+1]+g(4)\}, & m=10,\\
[f_2(m-4)+1]+\{2[f_2(m-5)+1]+g(m-6)\}-1, & m>10. \\
\end{cases}\\
&\leq f_2(m-1)+1.
\end{align*}
For $m > 10$, the third inequality holds because the orders of the vertex sets $X - \{y, y_1, y_2\}$ and $X - \{y, y_1, y_2, x_{11}\}$ are not contained in $\{6, 7, 10\}$. Specifically, if the equality holds, then $m=10$.

If there exists a component with vertices $y,y_1,y_2,y_3$ in $X_{11}-x_{11}$ (or $X_{12}-x_{12}$), then $(X,x)$ is  as shown in Fig. \ref{120} and $|X_{12}|\geq5$. By the induction hypothesis and Theorems~\ref{S.ChengB.Wu} and~\ref{Wang,Zhang,Tu,Xiong}, we have
\begin{align*}
\phi(X,x)&=\phi(X,x^0)+\phi(X,x^1)\\
&=\{\phi(X,x^0\bar{y})+\phi(X,x^0{y}^0)+[\phi(X,x^0{y}^1y_1)+\phi(X,x^0{y}^1y_2)+\phi(X,x^0{y}^1y_3)]\}+\phi(X,x^1)\\
&\leq \{f_2(m-5)+0+[2f_2(m-6)+\phi(X_{12},x_{12})\cdot f(m-|X_{12}|-7)]\}+4f(m-8)\\
&\leq
\begin{cases}
f_2(m-5)+2f_2(m-6)+5f(m-13)+4f(m-8), & |X_{11}|= 6, \\
f_2(m-5)+2f_2(m-6)+7f(m-14)+4f(m-8), & |X_{11}|= 7,10, \\
f_2(m-5)+2f_2(m-6)+f_2(|X_{12}|-1)f(m-|X_{12}|-7)+4f(m-8), & |X_{11}|\neq 6,7,10.\\
\end{cases}\\
&<f_2(m-1)+1.
\end{align*}
The third inequality holds since $X_{12}\not\cong P_3$ or $K_{1,3}$, which yields $\phi(X,x^0\bar{y})<f(m-5)$ and $\phi(X,x^0{y}^1y_1)=\phi(X,x^0{y}^1y_2)<f(m-6)$.

If $X_{11} - x_{11}$ and $X_{12} - x_{12}$ do not contain components of order $3$ or $4$, then $|X_{11}|\geq6$ and $|X_{12}|\geq6$. Let $Y$ be a component in $X_{11} - x_{11}$ of order at least $5$.
By Theorems~\ref{S.ChengB.Wu} and~\ref{Wang,Zhang,Tu,Xiong}, we have
\begin{align*}
\phi(X,x)&=\phi(X,x^0)+\phi(X,x^1)=
\phi(X_{11})\phi(X_{12})+\phi(Y)\phi(X-\{x,x_1,{x}_{11},{x}_{12}\}\cup Y)\\
&\leq
g(|X_{11}|)f_2(|X_{12}|)+g(|Y|)f_2(m-|Y|-4) \\
&\leq
\begin{cases}
\displaystyle\frac{15}{16}\cdot f_2(m-2)+5f_2(m-9), & |Y|= 5, \\
\displaystyle\frac{15}{16}\cdot f_2(m-2)+\displaystyle\frac{15}{16}f_2(m-4), &  |Y|\geq 6. \\
\end{cases}\\
&<f_2(m-1)+1.
\end{align*}
The third inequality holds since $X-\{x,x_1,{x}_{11},{x}_{12}\}\cup Y$ does not contain components of order $3$ or $4$, which yields $\phi(X-\{x,x_1,{x}_{11},{x}_{12}\}\cup Y)<f(m-|Y|-4)$.

\textbf{Subcase 2.2.} There exists $i\in [k]$ such that $d(x_i)=2$.

Without loss of generality, assume that $d(x_1) = 2$. Note that $|X_{11}| \geq 3$.

If $|X_{11}|=3$ (implying $k\geq2$), as shown in Fig.~\ref{18}, where $V(X_{11}) = \{x_{11}, y_1, y_2\}$, then we have 
\begin{align*}
\phi(X,x)&=\phi(X,x\bar{x}_{11})+\phi(X,x{x_{11}}^0)+[\phi(X,x{x_{11}}^1y_1)+\phi(X,x{x_{11}}^1y_2)+\phi(X,x{x_{11}}^1x_1)]\\
&=\phi(X-X_{11},x)+0+[2\phi(X-X_1,x)+0]\\
&\leq [f_2(m-4)+1]+0+\{2[f_2(m-5)+1]+0\}\\
&<f_2(m-1)+1.
\end{align*}
The third inequality holds because $T-X_{11}$ and $T-X_1$ satisfy the conditions in the lemma.

If $|X_{11}|=4$ (implying $k\geq2$), as shown in Fig.~\ref{19}, where $V(X_{11}) = \{x_{11},y,y_1, y_2\}$, then we have
\begin{align*}
\phi(X,x)&=\phi(X,x\bar{x}_{11})+\phi(X,x{x_{11}}^0)+\phi(X,x{x_{11}}^1)\\
&=[\phi(X,x\bar{x}_{11}\bar{y})+\phi(X,x\bar{x}_{11}{y}^0)+\phi(X,x\bar{x}_{11}{y}^1)]+\phi(X,x{x_{11}}^0)+\phi(X,x{x_{11}}^1)\\
&=[\phi(X,x\bar{x}_{11}\bar{y}x_1y_1y_2\bar{x}_{2}\cdots\bar{x}_{k})+0+2\phi(X-X_{11},x)]+2\phi(X-X_1,x)\\
&<\{f(m-7)+0+2[f_2(m-5)+1]\}+2[f_2(m-6)+1]\\
&\leq f_2(m-1)+1.
\end{align*}
The third inequality holds because $T-X_{11}$ and $T-X_1$ satisfy the conditions in the lemma.

Without loss of generality, we assume that for each $i \in [k]$, $X_i - x_i$ is neither $P_3$ nor $K_{1,3}$.
We now consider $|X_{11}|\geq5$.

\textbf{Subcase 2.2.1.} $k\geq 2$.

Observe that for every $i \in [2, k]$, $X_{11}$ is a component of $X - \left(N[x] \cup N(x_i)\right)$. Therefore, $\phi(X, x\bar{x}_1\cdots\bar{x}_{i-1}x_i\bar{x}_{i+1}\cdots\bar{x}_k) \leq f_2(m - k - |N(x_i)|).$ Furthermore, since $X_{11}-x_{11}$ does not contain components of order $3$ and the graph $X_i - x_i$ is neither $P_3$ nor $K_{1,3}$ for each $i \in [k]$, it follows that
$\phi(X, x{x}_1\bar{x}_{2}\cdots\bar{x}_{k}) \leq f_2(m - k - |N(x_1)|).$
By the induction hypothesis and Theorems~\ref{S.ChengB.Wu} and~\ref{Wang,Zhang,Tu,Xiong}, we have
{\small
\begin{align*}
\phi(X,x)&=\phi(X,x^0)+\phi(X,x^1)=\phi(X,x\bar{x}_1\cdots\bar{x}_{k}x_{11})+\sum\limits_{i=1}^{k}\phi(X,x\bar{x}_1\cdots\bar{x}_{i-1}x_{i}\bar{x}_{i+1}\cdots\bar{x}_{k})\\
&\leq \phi(X_{11},x_{11})f(m-k-1-|X_{11}|)+\sum\limits_{i=1}^{k}f_2(m-k-|N(x_i)|)\\
&\leq F_3(k):=
\begin{cases}
5f(m-k-7)+kf_2(m-k-2), & |X_{11}|= 6, \\
7f(m-k-8)+kf_2(m-k-2), & |X_{11}|= 7,10,\\
(k+1)f_2(m-k-2), & |X_{11}|\neq 6,7,10.
\end{cases}
\end{align*}
}
First, we can directly verify that $F_3(2)\geq F_3(3)$. Next, following a discussion similar to that for $F(k)$ and $F_2(k)$ above, we show that $F_3(k)-F_3(k+1)\geq 0$ for $k\geq3$. Hence, $\phi(X,x)\leq F_3(2)< f_2(m-1)+1$.
\begin{figure}[htbp]
\centering
 \begin{minipage}[t]{0.16\textwidth}
        \centering
        \includegraphics[width=\linewidth]{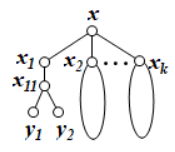}
       \caption{}
        \label{18}
    \end{minipage}
\hspace{0.3cm}
    \begin{minipage}[t]{0.15\textwidth}
        \centering
        \includegraphics[width=\linewidth]{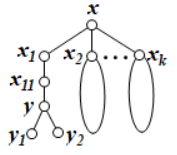}
        \caption{}
        \label{19}
    \end{minipage}
\hspace{0.3cm}
    \begin{minipage}[t]{0.13 \textwidth}
        \centering
        \includegraphics[width=\linewidth]{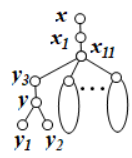}
        \caption{}
        \label{27}
    \end{minipage}
\hspace{0.3cm}
\begin{minipage}[t]{0.32\textwidth} 
        \includegraphics[width=\linewidth]{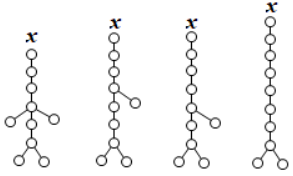} 
        \centering
        \caption{} 
        \label{L8_1} 
    \end{minipage}
\end{figure}

\textbf{Subcase 2.2.2.} $k = 1$.

The structure of $T$ together with $m\geq8$ implies that $X_{11}-x_{11}$ does not contain any components of order $2$ or $3$. If there exists a component of order $1$ in $X_{11}-x_{11}$, then by the induction hypothesis, we have 
{\small\begin{align*}
\phi(X,x)&=\phi(X,x^0)+\phi(X,x^1)=\phi(X_{11},x_{11})+\phi(X,x^1x_1\bar{x}_{11})\\
&\leq
\begin{cases}
f_2(m-3)+ f(m-4), & m=10, \\
 [f_2(m-3)+1]+ f(m-4), & m\neq10.\\
\end{cases}\\
&\leq f_2(m-1)+1.
\end{align*}}
Specifically, if equality holds, then $m=10$.

If there exists a component with vertices $y,y_1,y_2,y_3$ in $X_{11}-x_{11}$, then $(X,x)$ is as shown in Fig.~\ref{27}. We have
\begin{align*}
\phi(X,x)&=\phi(X,x\bar{y})+\phi(X,xy^0)+\phi(X,xy^1)\\
&=\phi(X-\{y,y_1,y_2\},x)+0+[2\phi(X-\{y,y_1,y_2,y_3\},x)+\phi(X,xy^1y_3x_1\bar{y}_1\bar{y}_2\bar{x}_{11})]\\
&\leq[ f_2(m-4)+1]+0+\{2[f_2(m-5)+1]+f(m-7)\}\\
&<f_2(m-1)+1.
\end{align*}
The third inequality holds because $T-\{y,y_1,y_2\}$ and $T-\{y,y_1,y_2,y_3\}$ satisfy the conditions in the lemma.

If $X_{11} - x_{11}$ does not contain any components of order  $1$ or  $4$, we may assume that $Y$ is a component in $X_{11} - x_{11}$ of order at least  $5$. If $d(x_{11})=2$ and $m = 10$, then the structure of $T$ implies that $(X, x)$ is as shown in Fig.~\ref{L8_1}, and the result follows by direct computation.

Otherwise, by Theorems~\ref{S.ChengB.Wu} and~\ref{Wang,Zhang,Tu,Xiong}, we have
\begin{align*}
\phi(X,x)&=\phi(X,x^0)+\phi(X,x^1)=\phi(X_{11},x_{11})+\phi(Y)\phi(X-\{x,x_1,{x}_{11}\}\cup Y)\\
&\leq 
\begin{cases}
[f_2(m-3)+1]+g(m-3), & d(x_{11})=2,~m\neq10, \\
[f_2(m-3)+1]+g(|Y|)f_2(m-|Y|-3), &  d(x_{11})\geq3. \\
\end{cases}\\
&\leq 
\begin{cases}
[f_2(m-3)+1]+g(m-3), & d(x_{11})=2,~m\neq10, \\
[f_2(m-3)+1]+5f_2(m-8), & d(x_{11})\geq3,~|Y|= 5, \\
[f_2(m-3)+1]+\frac{15}{16}f_2(m-3), &  d(x_{11})\geq3,~|Y|\geq6. \\
\end{cases}\\
&<f_2(m-1)+1.
\end{align*}
The second inequality holds since $X-\{x,x_1,{x}_{11}\}\cup Y$ does not contain components of order $3$ or $4$, which yields $\phi(X-\{x,x_1,{x}_{11}\}\cup Y)<f(m-|Y|-3)$.

We complete the proof.
\end{proof}

\section{Proof of Theorem \ref{main}}
For clarity, we restate our results here.
\renewcommand{\thetheorem}{3}
\begin{theorem}
For any tree $T$ of order $n\geq3$, then $\phi(T)\leq g(n),$ where
{\small
\[
g(n) =
\begin{cases}
n,  &  n=3,4,5,6,\\
3^{\frac{n-1}{3}}+\frac{n-1}{3}, &  n \equiv 1 \pmod{3},~n\geq7,\\
4\cdot 3^{\frac{n-5}{3}}+n-5, &  n \equiv 2 \pmod{3},~n\geq8, \\
16\cdot 3^{\frac{n-9}{3}}+3n-25, &  n \equiv 0 \pmod{3},~n\geq12~\text{and }~n\neq21, \\
19, &  n=9, \\
1349, & n=21.
\end{cases}
\]
}
The extremal trees that achieve this largest number are illustrated in Figure \ref{fig2}. 
\end{theorem}
\renewcommand{\thetheorem}

\textbf{Proof}. If $3\leq n\leq 6$, it is easy to prove that the result in Theorem \ref{main} holds true.
For $n\geq7$, we prove the theorem by induction on $n$.  Let $\widetilde{T}$ be a tree such that $\phi(\widetilde{T})=\max\{\phi({T}) : {T} \text{~is a tree and }|V({T})|=n\}$. Next, we prove that $\phi(\widetilde{T})\leq g(n)$. For $n=7,8,9$, it suffices by Lemmas~\ref{tu1} and~\ref{tu2} to verify Theorem~\ref{main} by a direct check of the following trees. By direct computation, we verify that Theorem \ref{main} holds for $n=7,8,9$.
\begin{figure}[htbp] 
    \centering
    \begin{minipage}[b]{0.55\textwidth} 
        \includegraphics[width=\linewidth]{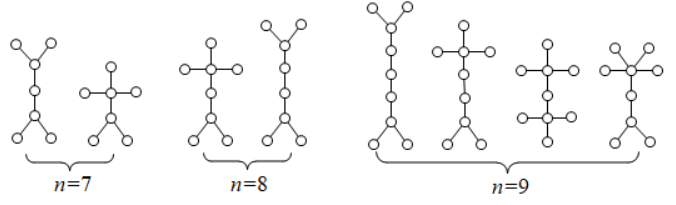} 
        \centering
    \end{minipage}
\end{figure}

Suppose that $n\geq10$ and that Theorem \ref{main} holds for trees of order at most $n-1$. Next, we prove that  Theorem \ref{main} holds for trees of order $n$. For $n \equiv 1 \pmod{3}$, the result follows directly from Theorem \ref{Wang,Zhang,Tu,Xiong}. Therefore, we consider only the case of $n \equiv 0,2 \pmod{3}$ here.

We now proceed by cases, depending on whether in the tree ${\widetilde{T}}$ there exists a vertex $u$ of degree 3 such that it has exactly two leaves and one nonleaf vertex with a degree of at least $3$. Therefore, the proof of Theorem \ref{main} is completed through the following two cases.

\noindent \textbf{Case 1.} There exists a vertex $u$ in the tree ${\widetilde{T}}$ such that $N(u)=\{x,y,z\},$ where $y$ and $z$ are leaves and $d_{{\widetilde{T}}}(x) \geq 3$.

Let $u\in V(\widetilde{T})$ such that $N(u)=\{x,y,z\},$ where $y$ and $z$ are leaves and $d_{{\widetilde{T}}}(x) \geq 3$. In fact, $x$ is not a support vertex of $\widetilde{T}$. Otherwise, suppose $x$ has a leaf neighbor $v$. By Lemma \ref{tu2}, we have $\phi(\widetilde{T})<\phi(\widetilde{T}-ux+uv)$, which contradicts $\phi(\widetilde{T})=\max\{\phi({T}) : {T} \text{~is a tree and }|V({T})|=n\}$. Let $X$ denote the component of ${\widetilde{T}}-u$ containing $x$. Recall that 
\begin{align}
\phi({\widetilde{T}})&=\phi({\widetilde{T}},\bar{u})+\phi(\widetilde{T},{u}^0)+\phi(\widetilde{T},{u}^1)\notag\\
&=\phi(X)+0+[\phi(\widetilde{T},{u}^1y)+\phi(\widetilde{T},{u}^1z)+\phi(\widetilde{T},{u}^1x)]\notag\\
&=\phi(X)+0+[2\phi(X-x)+\phi(X-N_X[x])].
\end{align}

In order to prove Case $1$, we proceed by proving the following Claim \ref{zw3}.
\begin{claim}\label{zw3}
Let $X$ be a tree of order $m$ (with $7 \leq m \leq n-3$), where every support vertex has a degree of at least $3$ and is adjacent to at most $2$ leaves, and let $x$ be a vertex of $X$ that is not a support vertex and has a degree of at least $2$. If $(X,x)$ is not equivalent to any of $(T_1,r_1)$ to $(T_4,r_4)$ depicted in Figure \ref{fig:T_1-T_4}, then the following inequality holds:

\[2\phi(X-x)+\phi(X-N[x])\leq h(m):
=\begin{cases}
2\cdot3^{\frac{m-1}{3}}+1 & \text{if } m \equiv 1 \pmod{3}, \\
8\cdot3^{\frac{m-5}{3}}+3 & \text{if } m \equiv 2 \pmod{3}, \\
32\cdot3^{\frac{m-9}{3}}+9 & \text{if } m \equiv 0 \pmod{3}.
\end{cases}\]
Equality holds if and only if the pair $(X,x)$ is as shown in Fig.~\ref{fig:iff-equality}. Moreover, for $m=21$, if $X$ is not as shown in Fig.~\ref{fig:iff-equality}, then $2\phi(X-x)+\phi(X-N[x])\leq2552$.

\begin{figure}[htbp]
    \centering 
    \includegraphics[width=0.63\textwidth]{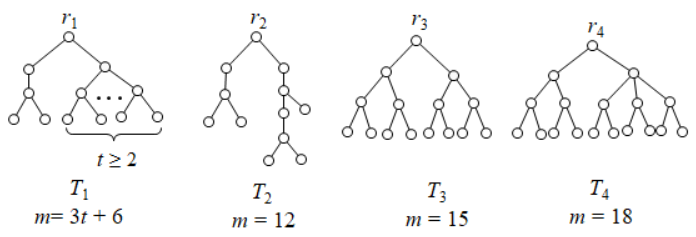}
    \caption{} 
    \label{fig:T_1-T_4} 
\end{figure}

\begin{figure}[htbp]
    \centering 
    \includegraphics[width=0.7\textwidth]{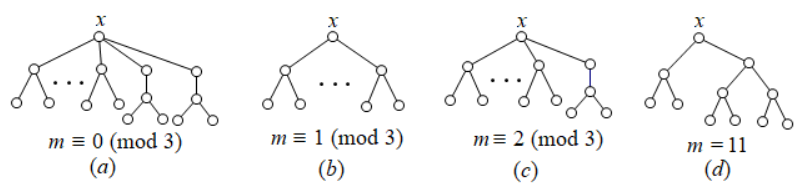}
    \caption{} 
    \label{fig:iff-equality} 
\end{figure}
\end{claim}
\begin{proof}[Proof of Claim~\ref{zw3}]
We prove Claim \ref{zw3} by induction on $m$. For $7 \leq m \leq 9$, all the trees satisfying the conditions of the Claim are listed in Figure~\ref{116}. Direct computation and verification of the equality case confirm that the result holds.

Suppose that the result holds for $10\leq |V(X)|\leq m-1$. We now prove this for $|V(X)|= m.$
 Let $N_X(x) = \{x_1, \dots, x_k\}$ and denote by $X_i$ the component of $X - x$ containing $x_i$ ($1\leq i\leq k$, $k\geq2$). Without loss of generality, assume that $X_1, \dots, X_k$ are ordered so that $|V(X_1)| \leq |V(X_2)| \leq \cdots \leq |V(X_k)|$.
\begin{figure}[htbp]
    \centering 
    \includegraphics[width=0.85\textwidth]{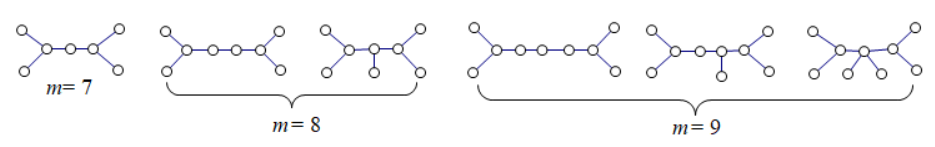}
    \caption{} 
    \label{116} 
\end{figure}

If $(X- X_1, x)$ is equivalent to one of the pairs in Fig.~\ref{fig:T_1-T_4}, we will show that the inequality $2\phi(X- x) + \phi(X- N[x]) < h(m)$ always holds. For example, if $(X - X_1, x)\sim (X_1, r_1)$, then $|V(X_2)|=4$. Moreover, we know that $X_1\cong K_1,P_3,K_{1,3}$ for $|V(X_1)|\leq|V(X_2)|=4$. By direct comparison, 
{\small
\begin{align*}
2\phi(X-x)+\phi(X-N[x])&=2\phi(X_1)\phi(X_2)\phi(X_3)+\phi(X_1-x_1)\phi(X_2-x_2)\phi(X_3-x_3)\\
&=2\phi(X_1)\cdot4\cdot\phi(X_3)+\phi(X_1-x_1)\cdot3\cdot\phi(X_3-x_3)\\
&=
\begin{cases}
2\cdot1\cdot4\cdot(3^{t}+t)+1\cdot3\cdot3^{t}, &X_{1}\cong K_1,~t=\frac{m-7}{3}, \\
2\cdot3\cdot4\cdot(3^{t}+t)+1\cdot3\cdot3^{t}, &X_{1}\cong P_3,~t=\frac{m-9}{3},\\
2\cdot4\cdot4\cdot(3^{t}+t)+3\cdot3\cdot3^{t}, & X_{1}\cong K_{1,3},~t=\frac{m-10}{3}, \\
\end{cases}\\
&<h(m),
\end{align*}
}
the second equality follows from Lemma \ref{tu0}. Specifically, for $m=21$, we have $X_1\cong P_3$ and $t=4$; hence, $2\phi(X-x)+\phi(X-N[x])\leq 339<2552$. All the other remaining cases can be analyzed and computed similarly. Therefore, we may assume that $(X - X_1, x)$ is not equivalent to any of $(X_1, r_1),\dots,(X_4, r_4)$. Let $\lambda=|V(X_1)|$. Since $x$ is not a support vertex of $X$, it follows that $\lambda\geq2$. We discuss the results according to the cases below.

\noindent \textbf{Case 1.} $\lambda=2$.

If $\lambda=2$, then $X_1\cong K_2$. Hence, $x_1$ is a support vertex of degree $2$ in $X$, contradicting the requirement that every support vertex of $X$ has a degree of at least $3$.

\noindent \textbf{Case 2.} $\lambda\geq3$.

\noindent \textbf{Subcase 2.1.} $k\geq3$.

Let $X'=X-X_1$ where $X'$ satisfies the conditions of Claim \ref{zw3}. Note that $\phi(X'-x)\leq f(m-\lambda-1)$ by Theorem \ref{S.ChengB.Wu}. By the induction hypothesis, $2\phi(X'-x)+\phi(X'-N[x])\leq h(m-\lambda)$. Thus, 
\begin{align}
&2\phi(X-x)+\phi(X-N[x])\notag\\
=&2\phi(X_1)\cdot\phi(X'-x)+\phi(X_1-x_1)\cdot\phi(X'-N[x])\notag\\
=&2\cdot[\phi(X_1)-\phi(X_1-x_1)]\cdot\phi(X'-x)+\phi(X_1-x_1)\cdot[2\phi(X'-x)+\phi(X'-N[x])]\\
\leq&2\cdot[\phi(X_1)-\phi(X_1-x_1)]\cdot f(m-\lambda-1)+\phi(X_1-x_1)\cdot h(m-\lambda).
\end{align}

\begin{figure}[htbp]
\centering
\begin{minipage}[t]{0.5\textwidth}
\includegraphics[width=\textwidth]{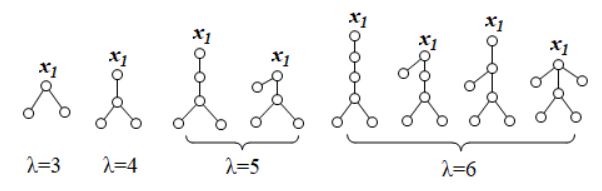}
    \caption{} 
    \label{L7-2} 
\end{minipage}
\end{figure}

For $\lambda=3,4,5,6$, from the structural conditions of $X$, we conclude that $(X_1, x_1)$ is as shown in Fig.~\ref{L7-2}. By direct comparison, we have
\begin{center}
\begin{minipage}{0.45\linewidth}
\[
\phi(X_1)\leq 
\begin{cases}
3, &\lambda=3, \\
4, &\lambda=4,\\
4, & \lambda=5, \\
6, & \lambda=6.
\end{cases}
\]
\end{minipage}
\hspace{-1cm}
\text{and}
\hspace{-1cm}
\begin{minipage}{0.45\linewidth}
\[
\phi(X_1-x_1)\leq
\begin{cases}
1, &\lambda=3, \\
3, &\lambda=4,\\
4, & \lambda=5, \\
4, & \lambda=6. 
\end{cases}
\tag{*}
\]
\end{minipage}

\end{center}
 Therefore, by $(5)$, $(6)$,
{\small
\begin{align*}
2\phi(X-x)+\phi(X-N[x])=&2\cdot[\phi(X_1)-\phi(X_1-x_1)]\cdot\phi(X'-x)+\phi(X_1-x_1)\cdot[2\phi(X'-x)+\phi(X'-N[x])]\\
\leq&2\cdot[\phi(X_1)-\phi(X_1-x_1)]\cdot f(m-\lambda-1)+\phi(X_1-x_1)\cdot h(m-\lambda)\\
\leq&
\begin{cases}
2\cdot(3-1)f(m-4)+1\cdot h(m-3), &\lambda=3, \\
2\cdot(4-3)f(m-5)+3\cdot h(m-4), &\lambda=4,\\
2\cdot(4-4)f(m-6)+4\cdot h(m-5), & \lambda=5, \\
2\cdot(6-4)f(m-7)+4\cdot h(m-6), & \lambda=6, 
\end{cases}\\
\leq& h(m).
\end{align*}
}
Equality holds if and only if $\lambda=3$ and $\phi(X'-x)=f(m-4)$, in which case one can verify that $(X,x)$ is as depicted in Fig.~\ref{fig:iff-equality}~$(a)(b)(c)$.

 For $m=21$, if the inequality holds strictly, that is, if either $\lambda=3$ and $\phi(X'-x)\neq f(m-4)$ or $\lambda>3$, then 
{\small\begin{align*}
2\phi(X-x)+\phi(X-N[x])
&\leq
\begin{cases}
2\cdot(3-1)f_2(17)+1\cdot h(18), &\lambda=3, \\
2\cdot(4-3)f(16)+3\cdot h(17), &\lambda=4,\\
2\cdot(4-4)f(15)+4\cdot h(16), & \lambda=5, \\
2\cdot(6-4)f(14)+4\cdot h(15), & \lambda=6, 
\end{cases}\\
&\leq 2552.
\end{align*}}
~~~Now, consider $\lambda \geq 7$. Note that $m\geq2\lambda+8$ since $m- \lambda-1\geq2|V(X_1)|=2\lambda$. By definition, we have 
\begin{align*}
h(m-\lambda)
&=
\begin{cases}
2\cdot3^{\frac{m-\lambda-1}{3}}+1=2\cdot f(m-\lambda-1)+1& \text{if } m-\lambda \equiv 1 \pmod{3}, \\
8\cdot3^{\frac{m-\lambda-5}{3}}+3=2\cdot f(m-\lambda-1)+3 & \text{if } m-\lambda \equiv 2 \pmod{3}, \\
32\cdot3^{\frac{m-\lambda-9}{3}}+9=2\cdot f(m-\lambda-1)+9 & \text{if } m-\lambda \equiv 0 \pmod{3}.
\end{cases}
\end{align*}
Note that $\phi(X_1) \leq g(\lambda)$ by induction and $\phi(X_1-x_1) \leq f(\lambda-1)$  by Theorem \ref{S.ChengB.Wu}. 
From $(6)$, 
\begin{align*}
2\phi(X-x)+\phi(X-N[x])\leq&2\cdot[\phi(X_1)-\phi(X_1-x_1)]\cdot f(m-\lambda-1)+\phi(X_1-x_1)\cdot h(m-\lambda)\\
\leq&2\phi(X_1)\cdot f(m-\lambda-1)+[h(m-\lambda)-2f(m-\lambda-1)]\cdot \phi(X_1-x_1)\\
\leq&2g(\lambda)\cdot f(m-\lambda-1)+[h(m-\lambda)-2f(m-\lambda-1)]\cdot f(\lambda-1)\\
=&
\begin{cases}
32g(\lambda)\cdot 3^{\frac{m-\lambda-9}{3}}+9f(\lambda-1)(=: A_0(m,\lambda)) & \text{if } m-\lambda \equiv 0 \pmod{3},\\
2g(\lambda)\cdot 3^{\frac{m-\lambda-1}{3}}+f(\lambda-1)(=: A_1(m,\lambda))& \text{if } m-\lambda \equiv 1 \pmod{3}, \\
8g(\lambda)\cdot 3^{\frac{m-\lambda-5}{3}}+3f(\lambda-1)(=: A_2(m,\lambda)) & \text{if } m-\lambda \equiv 2 \pmod{3}. 
\end{cases}
\end{align*}
 We shall prove that $A_i(m,\lambda)<h(m)$ for $i=0,1,2$.

We prove that $A_0(m,\lambda)<h(m)$ in detail. For $\lambda\geq7$ with $\lambda\neq9$,
{\small\begin{align*}
A_0(m,\lambda)-A_0(m,\lambda+3)&=32g(\lambda)\cdot 3^{\frac{m-\lambda-9}{3}}+9f(\lambda-1)-32g(\lambda+3)\cdot 3^{\frac{m-(\lambda+3)-9}{3}}-9f(\lambda+3-1)\\
&=32\cdot3^{\frac{m-\lambda-9}{3}}[g(\lambda)-\frac{1}{3}g(\lambda+3)]+9[f(\lambda-1)-f(\lambda+2)]\\
&\geq32\cdot3^{\frac{m-\lambda-9}{3}}-18\cdot3^{\frac{\lambda-1}{3}}\\
&\geq0.
\end{align*}}
Here, the second inequality holds since $g(\lambda)-\frac{1}{3}g(\lambda+3)\geq 1$ and $f(\lambda-1)-f(\lambda+2)\geq-2\cdot3^{\frac{\lambda-1}{3}}$, while the third inequality uses the condition $m\geq2\lambda+8$.

Therefore, 
\begin{align*}
A_0(m,\lambda)
&\leq
\begin{cases}
A_0(m,7)& \text{if } \lambda \equiv 1 \pmod{3} \\
A_0(m,8)& \text{if } \lambda \equiv 2 \pmod{3} \\
\max\{A_0(m,9),~ A_0(m,12)\}& \text{if } \lambda \equiv 0 \pmod{3}\\
\end{cases}
<h(m).
\end{align*}
 Specifically, $2\phi(X-x)+\phi(X-N[x])\leq \max\{A_0(21,9),~ A_0(21,12)\}= 2320<2552$ for $m=21$. 

The inequality $A_i(m,\lambda)<h(m)$ for $i=1,2$ can be proven by a similar discussion as above.

\noindent \textbf{Subcase 2.2.} $k=2$.

 Note that 
$2\phi(X-x)+\phi(X-N[x])=2\phi(X_1)\cdot\phi(X_2)+\phi(X_1-x_1)\cdot\phi(X_2-x_2)\leq 2\phi(X_1)g(m-\lambda-1)+\phi(X_1-x_1)\cdot f(m-\lambda-2)$ by induction and Theorem \ref{S.ChengB.Wu}.

\noindent \textbf{Subcase 2.2.1.} $\lambda=3,4,5,6$.

First, we discuss the case $\lambda=3,4,5,6$. From $(*)$ above, we have
{\small
\begin{align*}
2\phi(X-x)+\phi(X-N[x])
&\leq
\begin{cases}
2\cdot3\cdot g(m-4)+1\cdot f(m-5), &\lambda=3, \\
2\cdot4\cdot g(m-5)+3\cdot f(m-6), &\lambda=4,~m \equiv 1 \text{ or } 2 \pmod{3},\\
2\cdot4\cdot g(m-6)+4\cdot f(m-7), & \lambda=5, \\ 
2\cdot6\cdot g(m-7)+4\cdot f(m-8), & \lambda=6, 
\end{cases}\\
&\leq h(m).
\end{align*}
}
Equality holds if and only if $\lambda=3$, $m=11$ and $X_2-x_2\cong 2P_3$, as shown in Figure~\ref{fig:iff-equality}~$(d)$.
\begin{figure}[htbp]
\centering
\begin{minipage}[t]{0.5\textwidth}
\includegraphics[width=\textwidth]{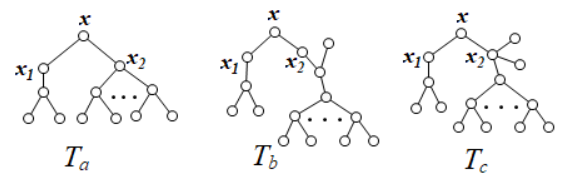}
    \caption{} 
    \label{T1-5} 
\end{minipage}
\end{figure}

If $\lambda = 4$ and $m \equiv 0 \pmod{3}$, then we have $X_2 - x_2 \not\cong \displaystyle\frac{m-6}{3} P_3$ owing to $(X,x) \not\sim (X_1,r_1)$. Consequently, by Theorems~\ref{S.ChengB.Wu} and~\ref{Wang,Zhang,Tu,Xiong}, it follows that $\phi(X_2 - x_2) \leq f_2(m-6)$. Furthermore, $\phi(X_2) = g(m-5)$ if and only if $X_2$ is the graph  shown in Fig.~\ref{fig2}, which means that $(X,x)$ is as shown in Fig.~\ref{T1-5} by induction. 
If $(X,x) \sim (X_a,x)$, this contradicts $(X,x) \not\sim (X_1,r_1)$. If $(X,x) \sim (X_b,x)$ for $m = 12$, this contradicts $(X,x) \not\sim (X_2,r_2)$. If $(X,x)\sim(X_b,x)$ for $m\neq12$, then $\phi(X_2 - x_2) < g(m-6)$ by induction. If $(X,x)\sim(X_c,x)$, then $\phi(X_2 - x_2) \leq g(m-8)$.  Hence,
{\small\begin{align*}
2\phi(X-x)+\phi(X-N[x])&=2\cdot4\cdot \phi(X_2)+3\cdot \phi(X_2-x_2)\\
&\leq
\begin{cases}
8g(m-5)+3[g(m-6)-1], &(X,x)\sim(X_b,x)~\text{for}~ m\neq12,\\
8g(m-5)+3g(m-8), &(X,x)\sim(X_c,x) , \\
8[g(m-5)-1]+3f_2(m-6), &\text{otherwise}. \\ 
\end{cases}\\
&\leq h(m).
\end{align*}}
~~~Specifically,  for $m=21$,
{\small
\begin{align*}
&2\phi(X-x)+\phi(X-N[x])\\
&\leq
\begin{cases}
2\cdot3\cdot g(17)+1\cdot f(16), &\lambda=3, \\
 \max\{8g(16)+3[g(15)-1],~8g(16)+3g(13),~8[g(16)-1]+3f_2(15)\},  &\lambda=4,\\
2\cdot4\cdot g(15)+4\cdot f(14), & \lambda=5, \\
2\cdot6\cdot g(14)+4\cdot f(13), & \lambda=6, 
\end{cases}\\
&\leq 2552.
\end{align*}}
\noindent \textbf{Subcase 2.2.2.} $\lambda\geq 7$.

 Note that $2\phi(X-x)+\phi(X-N[x])=2\phi(X_1)\phi(X_2)+\phi(X_1-x_1)\phi(X_2-x_2) \leq2g(\lambda)\cdot g(m-\lambda-1)+f(\lambda-1)\cdot f(m-\lambda-2)$ by induction and Theorem \ref{S.ChengB.Wu}. 

We prove the result according to different values of $m$ and $\lambda$ after modulo $3$. We prove the case $m \equiv 0 \pmod{3}$ and $\lambda  \equiv 0 \pmod{3}$ in detail. Other cases can be proven by a similar discussion. 

If $m \equiv 0 \pmod{3}$ and $\lambda  \equiv 0 \pmod{3}$, then $\lambda\geq9$ and $m\geq 21$. We now prove that 
$2g(\lambda)\cdot g(m-\lambda-1)+f(\lambda-1)\cdot f(m-\lambda-2)<h(m)$.

If $\lambda=9$ or $\lambda=21$, we have 
{\small
\begin{align*}
2g(\lambda)\cdot g(m-\lambda-1)+f(\lambda-1)\cdot f(m-\lambda-2)&=
\begin{cases}
2g(9)\cdot g(m-10)+f(8)\cdot f(m-11), &\lambda=9, \\
2g(21)\cdot g(m-22)+f(20)\cdot f(m-22), & \lambda=21, 
\end{cases}\\
&\leq 2552.
\end{align*}
}
 Specifically, $2\phi(X-x)+\phi(X-N[x])\leq 2172<2552$ for $m=21$.

Otherwise, $\lambda\geq12$ and $\lambda\neq21$. The inequalities $\lambda = |V(X_1)| \leq |V(X_2)| \leq m - \lambda - 1 \leq m - 13$, with $m, \lambda \equiv 0 \pmod{3}$, yield $\lambda \leq m - 15$. Note that 
 {\small\begin{align}
&2g(\lambda)\cdot g(m-\lambda-1)+f(\lambda-1)\cdot f(m-\lambda-2)\notag\\
=&2\cdot(16\cdot 3^{\frac{\lambda-9}{3}}+3\lambda-25)\cdot(4\cdot 3^{\frac{m-\lambda-6}{3}}+m-\lambda-6)+16\cdot 3^{\frac{\lambda-9}{3}}\cdot4\cdot 3^{\frac{m-\lambda-6}{3}}\notag\\
=&192\cdot3^{\frac{m-15}{3}}+(32m-32\lambda-192)\cdot3^{\frac{\lambda-9}{3}}+(24\lambda-200)\cdot3^{\frac{m-\lambda-6}{3}}+(6\lambda-50)\cdot(m-\lambda-6).
\end{align}}
 Let $B_1(m,\lambda)=(32m-32\lambda-192)\cdot3^{\frac{\lambda-9}{3}}$, $B_2(m,\lambda)=(24\lambda-200)\cdot3^{\frac{m-\lambda-6}{3}}$ and  $B_3(m,\lambda)=(6\lambda-50)\cdot(m-\lambda-6).$
We first prove that $B_1(m,\lambda)\leq B_1(m,m-15)$. This is because
\begin{align*}
B_1(m,\lambda+3)-B_1(m,\lambda)&=[32m-32(\lambda+3)-192]\cdot3^{\frac{\lambda+3-9}{3}}-(32m-32\lambda-192)\cdot3^{\frac{\lambda-9}{3}}\\
&=(64m-64\lambda-672)\cdot3^{\frac{\lambda-9}{3}}>0.
\end{align*}
Here, the last inequality holds since $\lambda\leq m-15$.

Similarly, we prove that $B_2(m,\lambda)\leq B_2(m,12)$.  This is because
\begin{align*}
B_2(m,\lambda)-B_2(m,\lambda+3)
&=(24\lambda-200)\cdot3^{\frac{m-\lambda-6}{3}}-[24(\lambda+3)-200]\cdot3^{\frac{m-\lambda-3-6}{3}}\\
&=(16\lambda-\frac{472}{3})\cdot3^{\frac{m-\lambda-6}{3}}>0.
\end{align*}
Here, the last inequality holds since $\lambda\geq 12$.

By the elementary inequality, we have $B_3(m,\lambda)\leq \displaystyle\frac{9m^2-258m+1849}{6}$. 

Hence, from $(7)$, we have
{\small\begin{align*}
&2g(\lambda)\cdot g(m-\lambda-1)+f(\lambda-1)\cdot f(m-\lambda-2)\\
=&192\cdot3^{\frac{m-15}{3}}+(32m-32\lambda-192)\cdot3^{\frac{\lambda-9}{3}}+(24\lambda-200)\cdot3^{\frac{m-\lambda-6}{3}}+(6\lambda-50)\cdot(m-\lambda-6)\\
\leq& 192\cdot3^{\frac{m-15}{3}}+[32m-32(m-15)-192]\cdot3^{\frac{m-15-9}{3}}+(24\cdot12-200)\cdot3^{\frac{m-12-6}{3}}+\frac{9m^2-258m+1849}{6}\\
<&h(m).
\end{align*}}
~~~Consequently, $2\phi(X-x)+\phi(X-N[x])\leq h(m)$, and the claim is true. 
\end{proof}
We proceed to prove Case~$1.$

If $(X,x)\sim (T_i,r_i)$ (see Figure~\ref{fig:T_1-T_4}) for each $1\leq i\leq4$, then $\widetilde{T}$ is as shown in Fig.~\ref{T2-4}. We will prove $\phi({\widetilde{T}})<g(n)$. For example, if $(X,x)\sim (T_1,r_1)$, then $\widetilde{T}$ is as shown in Fig.~\ref{T2-4}~$(a)$. From $(4)$, we have 
{\small\begin{align*}
\phi({\widetilde{T}})&=\phi(X)+2\phi(X-x)+\phi(X-N_X[x])\\
&=[\phi(X,\bar{v}_2)+\phi(X,{v_2}^0)+\phi(X,{v_2}^1)]+2\phi(X-x)+\phi(X-N_X[x])\\
&=\big\{[\phi({X}-{v_2})-\phi({X}-{v_2},\bar{x}\bar{u}_1\cdots\bar{u}_t)-\sum\limits_{i=1}^{t}\phi({X}-{v_2},\bar{x}\bar{u}_1\cdots\bar{u}_{i-1}{{u}_i}^0\bar{u}_{i+1}\cdots\bar{u}_{t})\\
&-\phi({X}-{v_2},{{x}}^0\bar{u}_1\cdots\bar{u}_{t})]+\phi(X,{v_2}^0\bar{x}\bar{u}_1\cdots\bar{u}_t)+\sum\limits_{i=1}^{t}\phi({X},{v_2}^1{{u}_i}\bar{x}\bar{u}_1\cdots\bar{u}_{i-1}\bar{u}_{i+1}\cdots\bar{u}_{t})\\
&+\phi(X,{v_2}^1{{x}}\bar{u}_1\cdots\bar{u}_{t})\big\}+2\phi(X-x)+\phi(X-N_X[x])\\
&=(4\cdot 3^t-1-0-2)+2+(4t+3)+[11\cdot3^{\frac{n-9}{3}}+\frac{8(n-9)}{3}] \\
&=15\cdot 3^{\frac{n-9}{3}}+4n-34<g(n),
\end{align*}}
the fifth equality holds for $t=\frac{n-9}{3}.$ Other cases can be proven by a similar discussion. 
\begin{figure}[htbp]
    \centering 
    \includegraphics[width=0.6\textwidth]{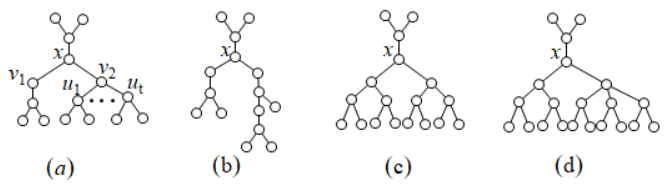}
    \caption{} 
    \label{T2-4} 
\end{figure}

Next, we consider $(X,x)\not\sim (T_i,r_i)$ for each $1\leq i\leq4$. By Lemma \ref{tu3}, we know that every support vertex of $\widetilde{T}$ has a degree of at least $3$ and is adjacent to at most $2$ leaves. Note that $d_X(v)=d_{\widetilde{T}}(v)$ for any support vertex $v$ of $X$. Then, $X$ is a tree satisfying the conditions of Claim \ref{zw3}. Since $x$ is  not a support vertex of $X$ and $d_X(x)\geq2$, by Claim \ref{zw3}, we have $2\phi(X-x)+\phi(X-N_X[x])\leq h(n-3)$. 
\begin{figure}[htbp]
    \centering 
    \includegraphics[width=0.9\textwidth]{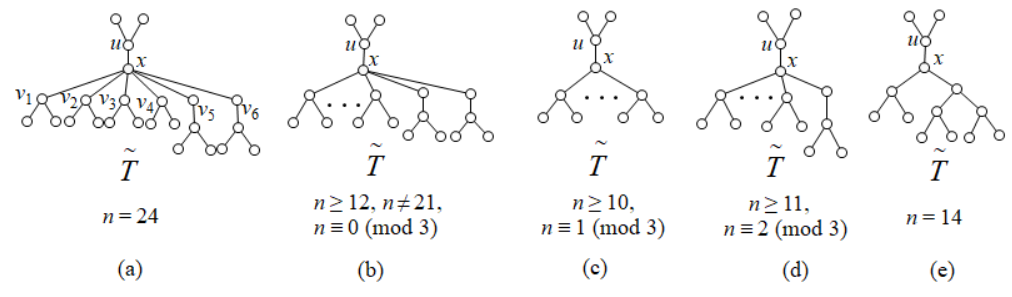}
    \caption{} 
    \label{118} 
\end{figure}

If $n = 12$, then $|V(X)| = 9$. Owing to the structure of $\widetilde{T}$, $X$ has no vertex with three leaves. Then, $\phi(X) < g(9)$ by induction. Therefore,
$\phi(\widetilde{T}) = \phi(X) +[ 2\phi(X - x) + \phi(X - N_X[x])] \leq (g(9) - 1) + h(9) = g(12)$.
Equality holds if and only if $2\phi(X - x) + \phi(X - N_X[x]) = h(9)$, which implies that $(X,x)$ is equivalent to Fig.~\ref{fig:iff-equality}$(c)$ and that $\widetilde{T}$ is isomorphic to Fig.~\ref{118}~$(c)$.

Next, we consider $n=24$ and $2\phi(X-x)+\phi(X-N_X[x])=h(21)$. Then, by Claim \ref{zw3}, ${\widetilde{T}}$ is as shown in Fig.~\ref{118}~$(a)$,  where $N_X(x)=\{v_1,v_2,v_3,v_4,v_5,v_6\}$. Hence,
\begin{align*}
\phi({\widetilde{T}})&=\phi(X)+h(21)\\
&= [\phi({X},\bar{x})+\phi({X},{x}^0)+\phi({X},{x}^1)]+h(21)\\
&= \big\{[\phi({X}-{x})-\phi({X}-{x},\bar{v}_1\cdots\bar{v}_6)-\sum\limits_{i=1}^{6}\phi({X}-{x},\bar{v}_1\cdots\bar{v}_{i-1}{{v}_i}^0\bar{v}_{i+1}\cdots\bar{v}_{6})]\\
&+\phi({X},{x}^0)+\sum\limits_{i=1}^{6}\phi({X},xv_i)\big\}+h(21)\\
&=[(3^4\cdot4^2-2^2-2\cdot2)+2^2+4\cdot3^2+2\cdot3]+h(21)\\
&=g(24)
\end{align*}
~~~Otherwise, we have 
{\small\begin{align*}
\phi({\widetilde{T}})&=\phi(X)+2\phi(X-x)+\phi(X-N_X[x])\\
&\leq
\begin{cases}
g(21)+2552, & n=24 ~\text{and}~ 2\phi(X-x)+\phi(X-N_X[x])<h(21),  \\
 g(n-3)+h(n-3), & n\neq24.\\
\end{cases}\\
&\leq g(n).
\end{align*}}
The equality holds if and only if $n\neq21$, $\phi(X)=g(n-3)$ and $2\phi(X-x)+\phi(X-N_X[x])=h(n-3)$, which implies that $\widetilde{T}$ is isomorphic to the tree in Fig.~\ref{118}~$(b)$,$(c)$,$(d)$,$(e)$.

\noindent \textbf{Case 2.} Every vertex with degree $3$ and exactly two leaf-neighbors has its nonleaf neighbor of degree at most $2$ in the tree $\widetilde{T}$.

Let $P =v_1v_2\cdots v_k$ be the longest path in $\widetilde{T}$. From the fact that $\widetilde{T}$ is not isomorphic to a star, we conclude that $k\geq4$. By Lemma \ref{tu3}, $\widetilde{T}-v_2$ has at most two leaf-neighbors. Thus, $d(v_2)\leq3$ by the choice of $P$. Moreover, $v_1v_2$ is not a component of $\widetilde{T}-v_3$  by Lemma \ref{tu3}. Thus, $d(v_2)\geq3$. Hence, $d(v_2)=3$. By the same proof, we have $d(v_{k-1})=3$. It must be that $k\geq5$; otherwise, $k=4$ forces $n=6$, contradicting the fact that $n \geq10$. We now know that both $v_2$ and $v_{k-1}$  are vertices of degree $3$ with exactly two leaf neighbors. On the basis of the conditions of Case $2$, we have $d(v_3)=d(v_{k-2})=2$. If $k=5$ or $6$, then $n=7~\text{or}~8$, which contradicts $n \geq10$. Hence, $k\geq7$, which implies that the diameter of $\widetilde{T}$ is at least $6$.

In $\widetilde{T}-v_4$, by the choice of $P$, there is at most one component with a diameter greater than $2$. If the component
of $\widetilde{T}-v_4$ has a diameter $2$, then the component is $K_{1,r},$ where $r\geq2$. By Lemma \ref{tu3}, we know that $r\leq3$. When $r=2$, $v_4$ is adjacent to the center of $K_{1,2}$, and when $r=3$, $v_4$ is adjacent to a leaf vertex of $K_{1,3}$. By the conditions of Case $2$, we know $r>2$. Therefore, the component with diameter $2$ is isomorphic to $K_{1,3}$. By Lemma \ref{tu3}, there are at most two components isomorphic to $K_1$ and no component isomorphic to $K_2$.

 Hence, we may choose a vertex $v$ such that $\widetilde{T} - v$ contains exactly $s$ components isomorphic to $K_1$, exactly $r+1$ components isomorphic to $K_{1,3}$, and at most one component that is not isomorphic to either $K_1$ or $K_{1,3}$, where $0 \leq s \leq 2, r \geq 0$. Under the above conditions, we choose a vertex $v$ satisfies $r$ so that it is as large as possible.

 By the choice of $\widetilde{T}$, $\widetilde{T}$ is as shown in Fig.~\ref{31} with $N(v)=\{v_{1},\dots,v_{r+1},y_{1},\dots,y_{s},x\}$ ($r\geq 0$, $s\in\{0,1,2\}$) and $X$ the component of $\widetilde{T}-v$ containing $x$.

If $X=\emptyset$, then $r\geq1$. We have
{\small\begin{align*}
\phi(\widetilde{T})&=\phi(\widetilde{T},\bar{v})+\phi(\widetilde{T},v^0)+\phi(\widetilde{T},v^1)\\
&=\phi(\widetilde{T}-v)-\phi(\widetilde{T}-v,\bar{v}\bar{v}_1\cdots\bar{v}_{r+1}\bar{y}_1\cdots\bar{y}_{s})-\sum\limits_{i=1}^{r+1}\phi(\widetilde{T}-v,\bar{v}\bar{v}_1\cdots\bar{v}_{i-1}{v_{i}}^0\bar{v}_{i+1}\cdots\bar{v}_{r+1}\bar{y}_1\cdots\bar{y}_{s})\notag\\
&-\sum\limits_{i=1}^{s}\phi(\widetilde{T}-v,\bar{v}\bar{v}_1\cdots\bar{v}_{r+1}\bar{y}_1\cdots\bar{y}_{i-1}{y_{i}}^0\bar{y}_{i+1}\cdots\bar{y}_{s})+\phi(\widetilde{T},{v}^0\bar{v}_1\cdots\bar{v}_{r+1}\bar{y}_1\cdots\bar{y}_{s})\notag\\
&+\sum\limits_{i=1}^{r+1}\phi(\widetilde{T},{v}{v_{i}})+\sum\limits_{i=1}^{s}\phi(\widetilde{T},{v}y_{i})\notag\\
&=
\begin{cases}
4^{r+1}-2^{r+1}-(r+1)\cdot2^r-0+2^{r+1}+(r+1)\cdot 3^{r}+0, & s=0,\\
4^{r+1}-0-0-2^{r+1}+0+ (r+1)\cdot 3^{r}+3^{r+1}, & s=1,\\
4^{r+1}-0-0-0 +0+ (r+1)\cdot 3^{r}+2\cdot3^{r+1}, & s=2.\\
\end{cases}
\end{align*}}

We shall prove that $\phi(\widetilde{T})\leq g(n)$ for $s=0,1,2$. We prove that $\phi(\widetilde{T})\leq g(n)$ for $s=0$ in detail. Note that $n=4r+5$. For $r=1,2,3$, a direct verification shows that $\phi(\widetilde{T})<g(n)=g(4r+5).$ For $r=4$, we have  $\phi(\widetilde{T})=4^{5}+5\cdot(3^4-2^4)=g(21).$ For $r \geq 5$, 
 we proceed by induction to show that $16\cdot 3^{\frac{n-9}{3}}+3n-25>4^{r+1}+(r+1)\cdot(3^r-2^r)$ for $n=4r+5$. We first verify the conclusion directly for  $r=5$.
We assume that $16\cdot 3^{\frac{4r-4}{3}}+12r-10>4^{r+1}+(r+1)\cdot(3^r-2^r)$ and now prove that $16\cdot 3^{\frac{4r}{3}}+12(r+1)-10>4^{r+2}+(r+2)\cdot(3^{r+1}-2^{r+1})$. Note that
{\small\begin{align*}
&16\cdot 3^{\frac{4r}{3}}+12(r+1)-10-4^{r+2}-(r+2)\cdot(3^{r+1}-2^{r+1})\\
\geq&3^{\frac{4}{3}}\cdot4^{r+1}+3^{\frac{4}{3}}\cdot(r+1)\cdot(3^{r}-2^{r})-3^{\frac{4}{3}}\cdot12r+10\cdot3^{\frac{4}{3}}+12(r+1)-10-4^{r+2}-(r+2)\cdot(3^{r+1}-2^{r+1})\\
\geq&[16\cdot(3^{\frac{4}{3}}-4)\cdot4^{r-1}-(3^{\frac{4}{3}}-1)\cdot12r]+[\frac{7}{2}\cdot(r+2)\cdot(3^{r}-2^{r})-(r+2)\cdot(3^{r+1}-2^{r+1})]+10\cdot3^{\frac{4}{3}}+2\\
>&0.
\end{align*}}
Here, the third inequality holds since $3^{\frac{4}{3}}\cdot(r+1)>\frac{7}{2}\cdot(r+2)$ for $r\geq5$, and the third inequality holds since $16\cdot(3^{\frac{4}{3}}-4)\cdot4^{r-1}-(3^{\frac{4}{3}}-1)\cdot12r>0$ for $r\geq5$ and $\frac{7}{2}\cdot(r+2)\cdot(3^{r}-2^{r})-(r+2)\cdot(3^{r+1}-2^{r+1})>0$.
The inequality $\phi(\widetilde{T})\leq g(n)$ for $s=1,2$ can be proven by a similar discussion as above.

\begin{figure}[htbp] 
    \begin{minipage}[t]{0.145\textwidth}
        \includegraphics[width=\linewidth]{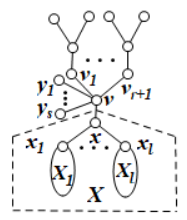}
       \caption{}
        \label{31}
    \end{minipage}
    \begin{minipage}[t]{0.116 \textwidth}
        \includegraphics[width=\linewidth]{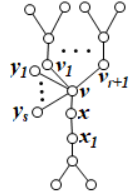}
        \caption{}
        \label{32}
    \end{minipage}
    \begin{minipage}[t]{0.116 \textwidth}
        \includegraphics[width=\linewidth]{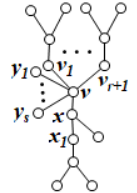}
        \caption{}
        \label{34}
    \end{minipage}
    \begin{minipage}[t]{0.099 \textwidth}
        \includegraphics[width=\linewidth]{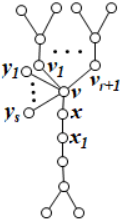}
        \caption{}
        \label{33}
    \end{minipage}
\begin{minipage}[t]{0.11\textwidth}
        \includegraphics[width=\linewidth]{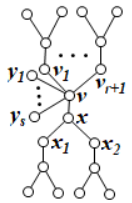}
        \caption{}
        \label{35}
    \end{minipage}
\begin{minipage}[t]{0.125\textwidth}
        \includegraphics[width=\linewidth]{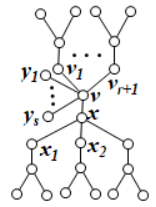}
       \caption{}
        \label{36}
    \end{minipage}
    \begin{minipage}[t]{0.097 \textwidth}
        \includegraphics[width=\linewidth]{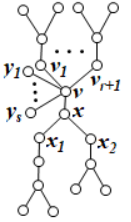}
        \caption{}
        \label{37}
    \end{minipage}
 \begin{minipage}[t]{0.103 \textwidth}
        \includegraphics[width=\linewidth]{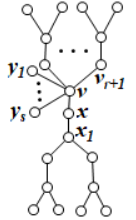}
        \caption{}
        \label{38}
    \end{minipage}
\end{figure}

We now consider $X\neq\emptyset$. Now $n=4(r+1)+|X|+s+1$. Owing to $\phi(\widetilde{T})=\phi(\widetilde{T},\bar{v})+\phi(\widetilde{T},v^0)+\phi(\widetilde{T},v^1)$, we estimate the values of $\phi(\widetilde{T},\bar{v})$, $\phi(\widetilde{T},v^0)$ and $\phi(\widetilde{T},v^1),$ respectively.

Note that
{\begin{align}
\phi(\widetilde{T},\bar{v})
=&\phi(\widetilde{T}-v)-\phi(\widetilde{T}-v,\bar{v}\bar{x}\bar{v}_1\cdots\bar{v}_{r+1}\bar{y}_1\cdots\bar{y}_{s})-\sum\limits_{i=1}^{r+1}\phi(\widetilde{T}-v,\bar{v}\bar{x}\bar{v}_1\cdots\bar{v}_{i-1}{v_{i}}^0\bar{v}_{i+1}\cdots\bar{v}_{r+1})\notag\\
&-\sum\limits_{i=1}^{s}\phi(\widetilde{T}-v,\bar{v}\bar{x}\bar{v}_1\cdots\bar{v}_{r+1}\bar{y}_1\cdots\bar{y}_{i-1}{y_{i}}^0\bar{y}_{i+1}\cdots\bar{y}_{s})-\phi(\widetilde{T}-v,\bar{v}{x}^0\bar{v}_1\cdots\bar{v}_{r+1})\notag\\
=&
\begin{cases}
4^{r+1}\phi(X)-2^{r+1}\phi(X,\bar{x})-(r+1)2^{r}\phi(X,\bar{x})-0-2^{r+1}\phi(X,{x}^0), & s=0,\\
4^{r+1}\phi(X)-0-0-2^{r+1}\phi(X,\bar{x})-0, & s=1,\\
4^{r+1}\phi(X)-0-0-0-0, & s=2, \\
\end{cases}
\end{align}}
and
{\small\begin{align}
\phi(\widetilde{T},{v}^0)+\phi(\widetilde{T},{v}^1)=&\phi(\widetilde{T},{v}^0\bar{x}\bar{v}_1\cdots\bar{v}_{r+1}\bar{y}_1\cdots\bar{y}_{s})+\sum\limits_{i=1}^{r+1}\phi(\widetilde{T},{v}{v_{i}})+\sum\limits_{i=1}^{s}\phi(\widetilde{T},{v}y_{i})+\phi(\widetilde{T},{v}x)\notag\\
\leq&
\begin{cases}
2^{r+1}\phi(X\cup\{v\},v^0\bar{x})+[(r+1)3^r\phi(X-x)+0+3^{r+1}\phi(X-N_X[x])], & s=0,\\
0+[(r+1)3^r\phi(X-x)+3^{r+1}\phi(X-x)+3^{r+1}\phi(X-N_X[x])], & s=1,\\
0+[(r+1)3^r\phi(X-x)+2\cdot3^{r+1}\phi(X-x)+3^{r+1}\phi(X-N_X[x])], & s=2.\\
\end{cases}
\end{align}}

 In fact, $\phi(X\cup\{v\},v^0\bar{x})\leq \phi(X,\bar{x})+\phi(X,x^1).$ Consider any $S\in MD(X\cup\{v\},v^0\bar{x})$. There exists $x_i\in N(x)$ such that $x_i\in S$. If $|S\cap N_X(x)|\geq 2$ or $|S\cap N_X(x)|= 1$ and $d_S(x_i)=1$, then $S\cap X\in MD(X,\bar{x})$.  If $|S\cap N_X(x)|= 1$ and $d_S(x_i)=0$, then $(S\cap X)\cup\{x\}\in MD(X,{x}^1)$. Hence $\phi(X\cup\{v\},v^0\bar{x})\leq \phi(X,\bar{x})+\phi(X,x^1).$

Therefore, from $\phi(X,\bar{x})=\phi(X)-\phi(X,x)$,  $\phi(X,{x}^0)=\phi(X,x)-\phi(X,x^1)$, $(8)$ and $(9)$, we have
{\small\begin{align}
\phi(\widetilde{T})&\leq
\begin{cases}
[4^{r+1}-(r+1)\cdot2^r]\phi(X)+(r-1)\cdot2^r\phi(X,{x})+2^{r+2}\phi(X,x^1)\\
+(r+1)\cdot3^r\phi(X-x)+3^{r+1}\phi(X-N_X[x]), & s=0,\\
(4^{r+1}-2^{r+1})\phi(X)+2^{r+1}\phi(X,{x})+(r+4)3^r\phi(X-x)+3^{r+1}\phi(X-N_X[x]), & s=1,\\
4^{r+1}\phi(X)+(r+7)3^r\phi(X-x)+3^{r+1}\phi(X-N_X[x]), & s=2. \\
\end{cases}
\end{align}}

 If $|V(X)|=5,6$, by the choice of $\widetilde{T}$, there are only three possible structures of $\widetilde{T},$ which are shown in Fig.~\ref{32},~\ref{34} and~\ref{33}. 

If $\widetilde{T}$ is as shown in Fig.~\ref{32}, then $n=4r+s+10$, $\phi(X)=4$, $\phi(X,x)=3$, $\phi(X-x)=4$, $\phi(X-N_X[x])=3$ and $\phi(X,x^1)=1$. Therefore, from $(10)$, we have
\begin{align*}
\phi(\widetilde{T})\leq
\begin{cases}
4^{r+2}+(4r+13)\cdot3^r-(r+3)\cdot2^r, & s=0,\\
4^{r+2} + (4r+25)\cdot 3^{r} - 2^{r+1}, & s=1,\\
4^{r+2} + (4r+37)\cdot 3^{r}, & s=2.\\
\end{cases}
\end{align*}
By direct calculations, similar to before, we obtain $\phi(\widetilde{T})<g(n)$ for $s=0,1,2$.

Similarly, if $\widetilde{T}$ is as shown in Fig.~\ref{34}, then $n=4r+s+11$, $\phi(X)=6$, $\phi(X,x)=4$, $\phi(X-x)=4$, $\phi(X-N_X[x])=3$ and $\phi(X,x^1)=4$. From $(10)$, we have $\phi(\widetilde{T})<g(n).$ 

If $\widetilde{T}$ is as shown in Fig.~\ref{33}, then $n=4r+s+11$, $\phi(X)=6$, $\phi(X,x)=5$, $\phi(X-x)=4$, $\phi(X-N_X[x])=4$ and $\phi(X,x^1)=3$. From $(10)$, we obtain $\phi(\widetilde{T})<g(n).$ 

Next, we consider $|V(X)|\geq 7$. Let $N_X(x) = \{v, x_1, \dots, x_l\}$ with $l \in \mathbb{Z}^+$. For each $i \in [l]$, let $X_i$ be the component of $\widetilde{T} - x$ containing $x_i$. We assume that $|X_{1}|\leq\cdots\leq|X_{l}|$. Note that $X-x$ does not contain $P_3$ as a component by $|X|\geq7$ and the condition of Case $2$.

By Theorem~\ref{S.ChengB.Wu}, we have $\phi(X-x) \leq f(|V(X)|-1).$ If $\phi(X-x)=f(|V(X)|-1)$, then $\widetilde{T}$ is as shown in Fig.~\ref{35} by Theorem~\ref{S.ChengB.Wu}. In this case, $n=4r+s+14$, $\phi(X)=18$, $\phi(X,x)=10$, $\phi(X-x)=16$, $\phi(X-N_X[x])=9$ and $\phi(X,x^1)=6$. From $(10)$, we obtain $\phi(\widetilde{T})<g(n).$ Hence, by Theorems~\ref{Wang,Zhang,Tu,Xiong}, we may assume that $\phi(X-x) \leq f_2(|V(X)|-1)$.

Note that $\phi(X-N_X[x])\leq\phi(X-x-x_1)$. By Theorem~\ref{S.ChengB.Wu}, we have $\phi(X-x-x_1)\leq f(|V(X)|-2)$. If $\phi(X-x-x_1)=f(|V(X)|-2)$, $\widetilde{T}$ is shown in Fig.~\ref{35},~\ref{36},~\ref{37} and~\ref{38} by Theorem~\ref{S.ChengB.Wu}. If $\widetilde{T}$ is as shown in Fig.~\ref{35}, then $\phi(\widetilde{T})<g(n)$ according to the preceding discussion. 

If $\widetilde{T}$ is as shown in Fig.~\ref{36}, then $n=4r+s+18$, $\phi(X)=79$, $\phi(X,x)=35$, $\phi(X-x)=64$, $\phi(X-N_X[x])=27$ and $\phi(X,x^1)=27$. From $(10)$, we have $\phi(\widetilde{T})<g(n).$ 

If $\widetilde{T}$ is as shown in Fig.~\ref{37}, then $n=4r+s+15$, $\phi(X)=26$, $\phi(X,x)=17$, $\phi(X-x)=16$, $\phi(X-N_X[x])=12$ and $\phi(X,x^1)=13$. From $(10)$, we have $\phi(\widetilde{T})<g(n).$ 

If $\widetilde{T}$ is as shown in Fig.~\ref{38}, then $n=4r+s+15$, $\phi(X)=35$, $\phi(X,x)=25$, $\phi(X-x)=18$, $\phi(X-N_X[x])=16$ and $\phi(X,x^1)=9$. From $(10)$, we have $\phi(\widetilde{T})<g(n).$ 

Hence, by Theorems~\ref{Wang,Zhang,Tu,Xiong}, we may assume that $\phi(X-x-x_1)\leq f_2(|V(X)|-2)$.

\textbf{Subcase 2.1.} $r\geq3$ or $rs\neq0$.

 By the induction hypothesis, we have $\phi(X)\le g(n-4r-5-s)$ for $s\in\{0,1,2\},$ and by Lemma~\ref{zw4}, we obtain $\phi(X,x)\leq f_{2}(|X|-1)+1$. Note that $\phi(X,x^1)\leq \phi(X-N_X[x]).$  Above all, from $(10)$, we have
{\small\begin{align}
\phi(\widetilde{T})
&\leq																												
\begin{cases}
[4^{r+1}-(r+1)\cdot2^r]g(n-4r-5)+[(r-1)\cdot2^r+(r+1)\cdot3^r]f_2(n-4r-6)\\
+(2^{r+2}+3^{r+1})f_2(n-4r-7)+(r-1)\cdot2^r, & s=0,\\
(4^{r+1}-2^{r+1})g(n-4r-6)+[2^{r+1}+(r+4)\cdot3^r]f_2(n-4r-7)\\
+3^{r+1}f_2(n-4r-8)+2^{r+1}, & s=1,\\
4^{r+1}g(n-4r-7)+(r+7)\cdot3^rf_2(n-4r-8)+3^{r+1}f_2(n-4r-9), & s=2. \\
\end{cases}
\end{align}}

For $r\geq3$, we prove the result $\phi(\widetilde{T})<g(n)$ according to different values of $n$ and $r$ after modulo $3$. We prove the case $n \equiv 0 \pmod{3}$ and $r  \equiv 0 \pmod{3}$ for $s=0$ in detail. Other cases can be proven by a similar discussion. 

In this case, $n=4(r+1)+|X|+1$ for $r\geq3$. We have $n\geq 24$ for $|X|\geq7$.
If $n\leq 33$, then the pair $(n,r)\in\{(24,3), (27,3), (30,3), (33,3)\}$. Direct computation confirms that $\phi(\widetilde{T}) < g(n)$ for these cases. Next, for $n \geq 34$, we have
\begin{align*}
g(n)-\phi(\widetilde{T})\geq&16\cdot 3^{\frac{n-9}{3}}+3n-25-[4^{r+1}-(r+1)\cdot2^r](3^{\frac{n-4r-6}{3}}+\frac{n-4r-6}{3})\\
&-[(r-1)\cdot2^r+(r+1)\cdot3^r]\cdot64\cdot3^{\frac{n-4r-18}{3}}
-(2^{r+2}+3^{r+1})\cdot15\cdot3^{\frac{n-4r-15}{3}}-(r-1)\cdot2^r\\
=&\{16-[81\cdot4^{r+1}+(64r+199)\cdot3^r-(17r-35)\cdot2^r]\cdot3^{-\frac{4r+9}{3}}\}\cdot3^{\frac{n-9}{3}}\\
&-\frac{1}{3}(n-4r-6)\cdot4^{r+1}+[\frac{1}{3}(n-4r-6)(r+1)-(r-1)]\cdot2^r+3n-25.
\end{align*}
 Let $w_1(r)=[81\cdot4^{r+1}+(64r+199)\cdot3^r-(17r-35)\cdot2^r]\cdot3^{-\frac{4r+9}{3}}$ and $w_2(r)=-\frac{1}{3}(n-4r-6)\cdot4^{r+1}$.
We first prove that $w_1(r+3) < w_1(r)$. This is because
\begin{align*}
w_1(r)-w_1(r+3)=&[81\cdot4^{r+1}+(64r+199)\cdot3^r-(17r-35)\cdot2^r]\cdot3^{-\frac{4r+9}{3}}\\
&-\{81\cdot4^{r+4}+[64(r+3)+199]\cdot3^{r+3}-[17(r+3)-35]\cdot2^{r+3}\}\cdot3^{-\frac{4r+21}{3}}\\
=&[17\cdot4^{r+1}+(\frac{128}{3}r+\frac{206}{3})\cdot3^r-(\frac{1241}{81}r-\frac{2963}{81})\cdot2^r]\cdot3^{-\frac{4r+9}{3}}>0.
\end{align*}

Moreover, note that $r\leq\frac{n-12}{4}$ for $|X|=n-4(r+1)-1\geq7$. Analyzing the derivative of $w_2(r)$ gives $w_2(r) \geq w_2(\frac{n-12}{4})= -2\cdot4^{\frac{n-8}{4}}.$ Observe that $n-4r-6\geq6$ for $n-4(r+1)-1\geq7$. Therefore, we have $\frac{1}{3}(n-4r-6)(r+1)-(r-1)\geq r+3>0$.

Overall, we have
\begin{align*}
g(n)-\phi(\widetilde{T})\geq&[16-w_1(3)]\cdot3^{\frac{n-9}{3}}-2\cdot4^{\frac{n-8}{4}}+[\frac{1}{3}(n-4r-6)(r+1)-(r-1)]\cdot2^r+3n-25\\
\geq&(16-\frac{31165}{2187})\cdot3^{\frac{n-9}{3}}-2\cdot4^{\frac{n-8}{4}}\\
\geq&\frac{3827}{2187}\cdot3^{\frac{n}{3}}-\frac{1}{8}\cdot4^{\frac{n}{4}}>0.
\end{align*}

For $rs\neq0$, from $(11)$, we have 
{\begin{align*}
\phi(\widetilde{T})&\leq
\begin{cases}
64 g(n-15)+81f_2(n-16)+27f_2(n-17), &r=2, s=2. \\
56 g(n-14)+8[f_2(n-15)+1]+54f_2(n-15)+27f_2(n-16), &r=2, s=1,\\
16 g(n-11)+24f_2(n-12)+9f_2(n-13), &r=1, s=2. \\
12 g(n-10)+4[f_2(n-11)+1]+15f_2(n-11)+9f_2(n-12), & r=1, s=1,\\
\end{cases}\\
&<g(n).
\end{align*}}
~~~\textbf{Subcase 2.3.} $1\leq r\leq2, s=0.$

If there exists a component with vertices $x_1,y,y_1,y_2,$ in $X-x$, then $\widetilde{T}$ is as shown in Fig.~\ref{r=1s=0-1}. Recall that $\phi(\widetilde{T})=\phi(\widetilde{T},\bar{y})+\phi(\widetilde{T},y^0)+\phi(\widetilde{T},y^1y_1)+\phi(\widetilde{T},y^1y_2)+\phi(\widetilde{T},y^1x_1).$ By Theorem~\ref{S.ChengB.Wu}, we have $\phi(\widetilde{T},y^1x_1)\leq[4^{r+1}+(r+1)\cdot 3^{r}-(r+1)\cdot2^r]f(n-4r-10)$. If $\phi(\widetilde{T},y^1x_1)=[4^{r+1}+(r+1)\cdot 3^{r}-(r+1)\cdot2^r]f(n-4r-10)$, then $X-x-X_1\cong K_{1,3}$ or $2K_{1,3}$ and $\phi(\widetilde{T})<g(n)$ follows by direct computation. Otherwise,  $\phi(\widetilde{T},y^1x_1)\leq[4^{r+1}+(r+1)\cdot 3^{r}-(r+1)\cdot2^r]f_2(n-4r-10)$ by Theorem~\ref{Wang,Zhang,Tu,Xiong}. Therefore, 
$\phi(\widetilde{T})\leq g(n-3)+0+2g(n-4)+[4^{r+1}+(r+1)\cdot 3^{r}-(r+1)\cdot2^r]f_2(n-4r-10)<g(n)$ for $r=1,2$.

Next, we consider that there exists a component of order $1$ in $X-x$. 

For $r=2$, from $(10)$ and Lemma~\ref{zw4}, we have 
\begin{align*}
\phi(\widetilde{T})&\leq52\phi(X)+4\phi(X,x)+27\phi(X-x)+16\phi(X,{x}^1)+27\phi(X-N_X[x]) \\
&\leq52g(n-13)+4[f_2(n-14)+1]+27f_2(n-15)+16lf(n-l-14)+27f(n-l-14) \\
&\leq52g(n-13)+4[f_2(n-14)+1]+27f_2(n-15)+59f(n-16)\\
&<g(n).
\end{align*}
The third inequality holds for two reasons. First, by Lemma \ref{1}~$(ii)$, $16lf(n-l-14)+27f(n-l-14)$ is decreases in $l$ for $l \geq 3$. Second, we have the inequality $59f(n-16) \geq 75f(n-17)$. Consequently, it follows that $16lf(n-l-14)+27f(n-l-14) \leq 59f(n-16)$.

For $r=1$,
by Theorem~\ref{S.ChengB.Wu}, we have $\phi(X-N_X[x])\leq f(n-l-10).$
If $\phi(X-N_X[x])=f(n-l-10),$ then $X-N_X[x]\cong K_{1,3}$ or $2K_{1,3}$. Since $n \equiv 0 \pmod{3}$ or $n \equiv 2 \pmod{3}$ , the structure of $\widetilde{T}$ is as shown in Fig.~\ref{r=1s=0-2}$~(a),(b),(c),(d)$. By direct calculations, we obtain $\phi(\widetilde{T})<g(n).$ Otherwise, for $r=1$, we have $\phi(X-N_X[x])\leq f_2(n-l-10)$ by Theorem~\ref{Wang,Zhang,Tu,Xiong}.

First, if $X-x$ has two components of order $1$, then $\phi(X-x)< f(n-12)$ since $X-x$ does not contain a component of order $3$ and $4$. From $(10)$, we have 
\begin{align*}
\phi(\widetilde{T})&\leq12\phi(X)+8\phi(X,x^1)+6\phi(X-x)+9\phi(X-N_X[x])\\
&\leq 12g(n-9)+[8(l-2)f(n-l-11)+16f_2(n-l-10)]+6f_2(n-12)+9f_2(n-l-10)\\
&\leq12g(n-9)+8f(n-14)+25f_2(n-13)+6f_2(n-12)<g(n).
\end{align*}
The third inequality holds because $8(l-2)f(n-l-11)+25f_2(n-l-10)$ decreases with increasing $l$ for $l\geq 3$ by Lemma \ref{1}$(ii)$.

Next, if $X-x$ has exactly one component of order $1$, then $X$ does not have the structure shown in Fig.~\ref{fig2}, and it follows that $\phi(X)<g(n-9).$ We have $\phi(X-x)< f(n-11)$ since $X-x$ does not contain a component of order $3$ and $4$. Note that $5\leq|X_2|\leq\cdots\leq|X_l|$. From $(10)$, we have
{\begin{align*}
\phi(\widetilde{T})\leq&12\phi(X)+8\phi(X,x^1)+6\phi(X-x)+9\phi(X-N_X[x])\\
\leq&12[g(n-9)-1]+[8(l-1)f(n-l-11)+8f_2(n-l-10)]\\
&+6\phi(X_2)\phi(X-x-X_1-X_2)+9f_2(n-l-10)\\
\leq&
\begin{cases}
12[g(n-9)-1]+8f(n-13)+17f_2(n-12)+6[g(n-11)-1], & l=2, n=24,\\
12[g(n-9)-1]+8f(n-13)+17f_2(n-12)+6g(n-11), & l=2, n \neq 24,\\
12[g(n-9)-1]+16f(n-14)+17f_2(n-13)+6\cdot5\cdot f_2(n-16), & l\geq3, |X_2|=5,\\
12[g(n-9)-1]+16f(n-14)+17f_2(n-13)+6\cdot\displaystyle\frac{15}{16}\cdot f_2(n-11), & l\geq3, |X_2|\geq6.\\
\end{cases}\\
<&g(n).
\end{align*}}
The third inequality holds because $X_2$ does not attain the upper bound $g(n-11)$ by the structure of $\widetilde{T}$ for $l=2, n=24$, and by Lemma~\ref{1}~$(ii)$, $8(l-1)f(n-l-11)+17f_2(n-l-10)$ decreases with increasing $l$ for $l\geq 3$.

\begin{figure}[htbp]
\centering
\begin{minipage}[t]{0.115 \textwidth}
        \includegraphics[width=\linewidth]{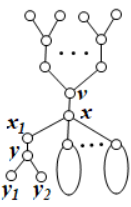}
        \caption{}
        \label{r=1s=0-1}
    \end{minipage}
\hspace{0.4cm}
\begin{minipage}[t]{0.47\textwidth}
        \includegraphics[width=\linewidth]{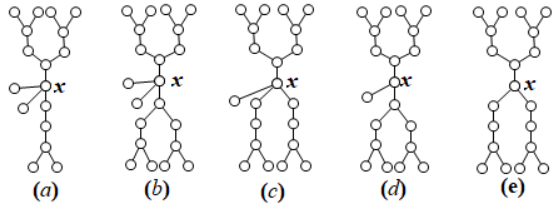}
        \caption{}
        \label{r=1s=0-2}
    \end{minipage}
\hspace{0.4cm}
\begin{minipage}[t]{0.12\textwidth}
        \includegraphics[width=\linewidth]{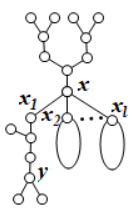}
        \caption{}
        \label{r=1s=0-3}
    \end{minipage}
\hspace{0.4cm}
\begin{minipage}[t]{0.065\textwidth}
        \includegraphics[width=\linewidth]{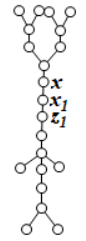}
        \caption{}
        \label{r=1s=0-4}
    \end{minipage}
\end{figure}

Next, we consider that $X - x$ does not contain any components of order  $1$ or  $4$. Note that $5\leq|X_1|\leq\cdots\leq|X_l|$.

For $r=2$, we have $\phi(\widetilde{T})\leq52\phi(X)+4\phi(X,x)+27\phi(X_1)\phi(X-X_1-x)+16\phi(X,{x}^1)+27\phi(X-N_X[x])$. If either $l \geq 2$, or $l = 1$ and $|X| \neq 7$, then the tree $X$ does not have the structure shown in Fig.~\ref{fig2}, consequently, $\phi(X) < g(n-13)$. By Theorem~\ref{S.ChengB.Wu}, we have $\phi(X-N_X[x])\leq f(n-15)$. If $\phi(X-N_X[x])= f(n-15)$, then $\widetilde{T}$ is shown in Fig.~\ref{38}. On the basis of the preceding discussion, we have  $\phi(\widetilde{T})<g(n).$ Thus, $\phi(X-N_X[x])< f(n-15)$. From $(10)$, by the induction hypothesis and Lemma~\ref{1}~$(ii)$,~\ref{zw4}, we have 
{\small\begin{align*}
&\phi(\widetilde{T})\\
\leq&
\begin{cases}
52\phi(X)+4\phi(X,x)+27g(n-14)+16f(n-16)+27f_2(n-15), &l=1,\\
52\phi(X)+4\phi(X,x)+27g(|X_1|)f_2(n-|X_1|-14)+16lf(n-l-15)+27f(n-l-14), &l\geq2. \\
\end{cases}\\
\leq&
\begin{cases}
52g(n-13)+4[f_2(n-14)+1]+27g(n-14)+16f(n-16)+27f_2(n-15), &l=1,~ |X|=7,\\
52[g(n-13)-1]+4[f_2(n-14)+1]+27g(n-14)+16f(n-16)+27f_2(n-15), &l=1,~ |X|\neq7,\\
52[g(n-13)-1]+4[f_2(n-14)+1]+27\cdot5f_2(n-19)+32f(n-17)+27f(n-16), &l\geq2,~ |X_1|=5,\\
52[g(n-13)-1]+4[f_2(n-14)+1]+27\cdot\displaystyle\frac{15}{16}f_2(n-14)+32f(n-17)+27f(n-16), &l\geq2,~|X_1|\geq6. \\
\end{cases}\\
<&g(n).
\end{align*}}
The second inequality holds because $5\leq|X_{2}|\leq\cdots\leq|X_{l}|$~(hence, $\phi(X-X_1-x)\leq f_2(n-|X_1|-14)$ by Theorem~\ref{S.ChengB.Wu} and Theorem~\ref{Wang,Zhang,Tu,Xiong}), and by Lemma \ref{1}$(ii)$, $16lf(n-l-15)+27f(n-l-14)$ decreases with increasing $l$ for $l\geq 2$.

For $r=1$ and $l\geq2$, we have $\phi(\widetilde{T})\leq12\phi(X)+8\phi(X,x^1)+6\phi(X_1)\phi(X-X_1-x)+9\phi(X-N_X[x]).$ In this case, $X$ does not have the structure shown in Fig.~\ref{fig2}, which implies that $\phi(X)<g(n-9)$. By Theorem~\ref{S.ChengB.Wu} and the condition $5 \leq |X_2| \leq \cdots \leq |X_l|$, we have $\phi(X - X_1 - x) < f(n - |X_1| - 10).$ By Theorem~\ref{S.ChengB.Wu}, we have $\phi(X-N_X[x])\leq f(n-l-10).$
If $\phi(X-N_X[x])=f(n-l-10),$ then $X-N_X[x]\cong 2K_{1,3}$, meaning that $\widetilde{T}$ is as shown in Fig.~\ref{r=1s=0-2}~$(e)$. By direct calculations, we obtain $\phi(\widetilde{T})<g(n).$ We may therefore assume that $\phi(X - N_X[x]) \leq f_2(n - l - 10)$. Note that $n\geq 2|X_1|+10$. 

If $|X_1|\geq10$, then $n\geq 30$. 
  From $(10)$, by Lemma~\ref{1}~$(ii),$ we have
{\small\begin{align*}
\phi(\widetilde{T})&\leq12[g(n-9)-1]+8lf(n-l-11)+6\phi(X_1)f_2(n-|X_1|-10)+9f_2(n-l-10)\\
&\leq12[g(n-9)-1]+16f(n-13)+6\cdot\displaystyle\frac{15}{16}f_2(n-10)+9f_2(n-12)\\
&<g(n).
\end{align*}}
The second inequality holds because $8l f(n - l - 11) + 9f_2(n - l - 10)$ is decreasing in $l$ for $l \geq 2$.

Next, we consider that $|X_1|=5,6,7,8,9$. If $|X_1| = 7$ and $\phi(X_1)= g(7)$, then the structure of $\widetilde{T}$ is as shown in Fig.~\ref{r=1s=0-3}. Thus, $\phi(\widetilde{T})=\phi(\widetilde{T},\bar{y})+\phi(\widetilde{T},{y}^0)+\phi(\widetilde{T},{y}^1)\leq g(n-3)+0+[2g(n-4)+g(n-6)]<g(n).$ Otherwise, from $(10)$, by Lemma~\ref{1}~$(ii),$ we have
{\small\begin{align*}
\phi(\widetilde{T})&\leq
\begin{cases}
12[g(15)-1]+8\cdot2\cdot f(n-l-11)+6[g(7)-1]^2+9f_2(12), & |X_1|=7, n=24,\\
12[g(n-9)-1]+8lf(n-l-11)+6\phi(X_1)f_2(n-|X_1|-10)+9f_2(n-l-10), & \text{Otherwise}.\\
\end{cases}\\
&\leq
\begin{cases}
12[g(15)-1]+8\cdot2\cdot f(11)+6[g(7)-1]^2+9f_2(12), & |X_1|=7, n=24,\\
12[g(n-9)-1]+16f(n-13)+6[g(|X_1|)-1] f_2(n-|X_1|-10)+9f_2(n-12), & |X_1|=5,7(n\neq24),\\
12[g(n-9)-1]+16f(n-13)+6g(|X_1|)f_2(n-|X_1|-10)+9f_2(n-12), & |X_1|=6,8,9,\\
\end{cases}\\
&<g(n).
\end{align*}}
The third inequality holds for four reasons. First, owing to the structure of $\widetilde{T}$, $X_1$ does not have the structure shown in Fig.~\ref{fig2} for $|X_1|=5$. Second, for $|X_1| = 7$ and $n = 24$, it follows that $l = 2$ and $|X_1| = |X_2| = 7$. 

We now consider $l=1$. If $d(x_1)\geq3$, let $Y$ be a component of $X_1-x_1$ with the minimum order. 
From $(10)$, by Lemma \ref{1}~$(ii)$, we have
{\small\begin{align*}
\phi(\widetilde{T})&\leq12\phi(X)+8\phi(X,x^1)+6\phi(X-x)+9\phi(Y)\phi(X-N[x]-Y)\\
&\leq
\begin{cases}
12g(n-9)+8f(n-13)+6g(n-10)+9f(n-12), & |Y|=1,\\
12g(n-9)+8f(|Y|-1)f(n-|Y|-12)+6g(n-10)+9g(|Y|)f(n-|Y|-11), & |Y|=4,5,\\
12g(n-9)+8f(n-13)+6g(n-10)+9\cdot\displaystyle\frac{15}{16}f_2(n-11), & |Y|\geq6.\\
\end{cases}\\
&<g(n).
\end{align*}}
~~~~If $d(x_1)=2$, let $N(x_1)=\{x,z_{1}\}$. If $d(z_{1})\geq3$, let $Z$ be a component of $X_1-\{x_1,z_{1}\}$ with the minimum order. 
From $(10)$, by Lemma \ref{1}~$(ii)$, we have
{\small\begin{align*}
\phi(\widetilde{T})&\leq12\phi(X)+8\phi(Z)\phi(X-\{x,x_1,z_1\}-Z)+6\phi(X-x)+9\phi(X-N_X[x])\\
&\leq
\begin{cases}
12[g(n-9)-1]+8f(n-13)+6g(n-10)+9g(n-11), & |Z|=1,\\
12[g(n-9)-1]+32f(n-16)+6[g(n-10)-1]+9g(n-11), & |Z|=4,~n\neq21,\\
12[g(n-9)-1]+32g(n-16)+6[g(n-10)-1]+9g(n-11), & |Z|=4,~n=21,\\
12[g(n-9)-1]+32f(n-17)+6g(n-10)+9g(n-11), & |Z|=5,\\
12[g(n-9)-1]+8\cdot\displaystyle\frac{15}{16}f_2(n-12)+6g(n-10)+9g(n-11), & |Z|\geq6.\\
\end{cases}\\
&<g(n).
\end{align*}}
The second inequality holds for two reasons: First, owing to the structure of $\widetilde{T}$, $X$ does not attain the upper bound $g(n-9)$; second, for $|Z|=4$ and $n=21$, it follows that $d(z_{1})=3$. 

Next, we consider $d(z_{1})=2$. If $n=21$ and $\phi(X,x^1)=g(n-12)$, then $\widetilde{T}$ is as shown in  Fig. \ref{r=1s=0-4}. By direct calculations, $\phi(\widetilde{T})<g(21).$ Hence, we have 
{\small\begin{align*}
\phi(\widetilde{T})&\leq12\phi(X)+8\phi(X,x^1)+6\phi(X-x)+9\phi(X-N_X[x])\\
&\leq
\begin{cases}
12[g(n-9)-1]+8[g(n-12)-1]+6[g(n-10)-1]+9[g(n-11)-1], & n=21,\\
12[g(n-9)-1]+8g(n-12)+6g(n-10)+9g(n-11), & n\neq21,\\
\end{cases}\\
&<g(n).
\end{align*}}
The second inequality holds for two reasons. First, owing to the structure of $\widetilde{T}$, the tree $X$ does not attain the upper bound $g(n-9)$. Second, for $n=21$, it follows that neither $X-x$ nor $X-N_X[x]$ attains the upper bounds $g(n-10)$ and $g(n-11)$, respectively.

\begin{figure}[htbp]
\centering
 \begin{minipage}[t]{0.066 \textwidth}
        \includegraphics[width=\linewidth]{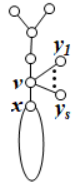}
        \caption{}
        \label{81}
    \end{minipage}
\hspace{0.8cm}
\begin{minipage}[t]{0.103 \textwidth}
        \includegraphics[width=\linewidth]{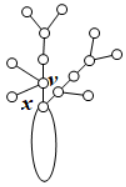}
        \caption{}
        \label{83}
    \end{minipage}
\hspace{0.3cm}
    \begin{minipage}[t]{0.103 \textwidth}
        \includegraphics[width=\linewidth]{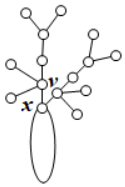}
        \caption{}
        \label{84}
    \end{minipage}
\hspace{0.3cm}
    \begin{minipage}[t]{0.105 \textwidth}
        \includegraphics[width=\linewidth]{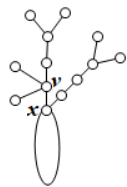}
        \caption{}
        \label{82}
    \end{minipage}
\hspace{0.3cm}
\begin{minipage}[t]{0.095 \textwidth}
        \includegraphics[width=\linewidth]{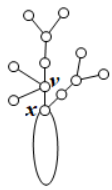}
        \caption{}
        \label{r=0s=2-1}
    \end{minipage}
\hspace{0.3cm}
    \begin{minipage}[t]{0.105 \textwidth}
        \includegraphics[width=\linewidth]{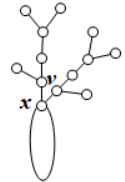}
        \caption{}
        \label{86}
    \end{minipage}
\hspace{0.3cm}
    \begin{minipage}[t]{0.102 \textwidth}
        \includegraphics[width=\linewidth]{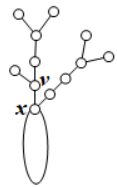}
        \caption{}
        \label{85}
    \end{minipage}
\end{figure}

\textbf{Subcase 2.3.} $r=0.$

In this case, we assume that $v$ is the vertex $v_4$ on the longest path that was discussed at the beginning of Case $2.$ See  Fig.~\ref{81}.

\textbf{Subcase 2.3.1.} $s\neq 0.$

If $d(x)\geq 3$, then each component $S$ of $\widetilde{T}-x$ disjoint from $P$ satisfies $\max\{d(u,x):u\in V(S)\}\in\{1,3,4\}.$ If there exists a component $S$ of $\widetilde{T}-x$ disjoint from $P$ satisfying $\max\{d(u,x):u\in V(S)\}\in\{3,4\},$ then only seven graphs remain to be verified; see Fig.~\ref{83},~\ref{84},~\ref{82},~\ref{r=0s=2-1},~\ref{86},~\ref{85} and~\ref{r=0s=1-3}. Note that $\phi(\widetilde{T})=\phi(\widetilde{T},\bar{v}\bar{v}_1)+\phi(\widetilde{T},\bar{v}{v_1})+\phi(\widetilde{T},v^0)+\phi(\widetilde{T},v^1\bar{x})+\phi(\widetilde{T},v^1{x}).$ When $\widetilde{T}$ is shown in Fig.~\ref{r=0s=2-1}, by Theorem~\ref{S.ChengB.Wu}, we have $\phi(\widetilde{T},v^1\bar{x})\leq28f(n-12)$. On the basis of the preceding discussion, we have that $X - x$ does not contain component isomorphic to $P_3$, and $X - x\not\cong rK_{1,3}$ where $r\geq2$. Thus $\phi(\widetilde{T},v^1\bar{x})<28f(n-12)$. Therefore,
 by Theorems~\ref{S.ChengB.Wu} and \ref{Wang,Zhang,Tu,Xiong}, we have
{\small\begin{align*}
\phi(\widetilde{T})&\leq
\begin{cases}
2g(n-7)+2g(n-7)+0+42f(n-14)+12f(n-15),&\text{if}~\widetilde{T} ~\text{is as shown in Fig.}~\ref{83},\\
2g(n-7)+2g(n-7)+0+77f(n-15)+12f(n-16), &\text{if}~\widetilde{T} ~\text{is as shown in Fig.}~\ref{84},\\
2g(n-7)+2g(n-7)+0+28f(n-13)+12f(n-12), &\text{if}~\widetilde{T} ~\text{is as shown in Fig.}~\ref{82},\\
2g(n-7)+2g(n-7)+0+28f_2(n-12)+9f(n-13), &\text{if}~\widetilde{T} ~\text{is as shown in Fig.}~\ref{r=0s=2-1},\\
2[f_2(n-7)+1]+2g(n-6)+0+24f(n-13)+12f(n-14),&\text{if}~\widetilde{T} ~\text{is as shown in Fig.}~\ref{86},\\
2[f_2(n-7)+1]+2g(n-6)+0+16f(n-12)+12f(n-13), &\text{if}~\widetilde{T} ~\text{is as shown in Fig.}~\ref{85}.\\
\end{cases}\\
&<g(n).
\end{align*}}
~~~~If $\widetilde{T}$ is as shown in Fig.~\ref{r=0s=1-3} and $n=18$, then $\widetilde{T}$ is as shown in Fig.~\ref{r=0s=1-2}~$(a)$,~$(b)$,~$(c)$,~$(d)$. By direct calculations, we obtain $\phi(\widetilde{T})<g(n).$ 
If $n\neq18$, similar to the discussion for Fig.~\ref{r=0s=2-1}, we have $\phi(X-N[y]-x)< f(n-11)$. Then, $\phi(\widetilde{T})=\phi(\widetilde{T},\bar{y})+\phi(\widetilde{T},{y}^0)+\phi(\widetilde{T},{y}^1)\leq g(n-3)+0+[2g(n-4)+6f_2(n-11)]<g(n)$.
\begin{figure}[htbp]
\centering
 \begin{minipage}[t]{0.098 \textwidth}
        \includegraphics[width=\linewidth]{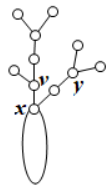}
        \caption{}
        \label{r=0s=1-3}
    \end{minipage}
\hspace{0.6cm}
 \begin{minipage}[t]{0.4 \textwidth}
        \includegraphics[width=\linewidth]{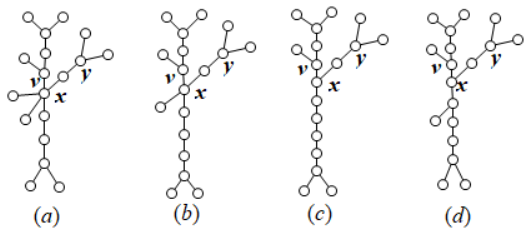}
        \caption{}
        \label{r=0s=1-2}
    \end{minipage}
\end{figure}

Otherwise, the components of $\widetilde{T}-x$ disjoint from $P$ are isomorphic to $K_1$. By Theorems~\ref{S.ChengB.Wu} and~\ref{Wang,Zhang,Tu,Xiong}, we have
{\small\begin{align*}
\phi(\widetilde{T})=&\phi(\widetilde{T},\bar{v}\bar{v}_1)+\phi(\widetilde{T},\bar{v}{v_1})+\phi(\widetilde{T},v^0)+\phi(\widetilde{T},v^1\bar{x})+\phi(\widetilde{T},v^1{x})\\
\leq&
\begin{cases}
2g(n-7)+2g(n-7)+0+7g(n-10)+3f(n-11),  &s=2,~x ~\text{has two leaf neighbors},\\
2g(n-7)+2g(n-7)+0+7g(n-9)+3f(n-10), &s=2,~ x ~\text{has one leaf neighbor},\\
2[f_2(n-7)+1]+2g(n-6)+0+4g(n-9)+3f(n-10),  &s=1,~x ~\text{has two leaf neighbors},\\
2[f_2(n-7)+1]+2g(n-6)+0+4g(n-8)+3f(n-9), &s=1,~ x ~\text{has one leaf neighbor}.\\
\end{cases}\\
<&g(n).
\end{align*}}
~~~If $d(x)=2$, let $N(x)=\{v,x_1\}.$  Let $Y$ be a component of $X-\{x,x_1\}$ with the minimum order. Thus,
{\small
\begin{align*}
&\phi(\widetilde{T})\\
=&\phi(\widetilde{T},\bar{v}\bar{v}_1)+\phi(\widetilde{T},\bar{v}{v_1})+\phi(\widetilde{T},v^0)+\phi(\widetilde{T},v^1\bar{x})+\phi(\widetilde{T},v^1{x})\\
\leq&
\begin{cases}
2g(n-7)+2g(n-7)+0+7g(n-8)+3f(n-9)-4, &s=2,~ n=15,\\
2g(n-7)+2g(n-7)+0+7g(n-8)+3f(n-9), &s=2,~  n\neq15,\\
2[f_2(n-7)+1]+2[g(n-6)-1]+0+4[g(n-7)-1]+3[g(n-8)-1], &s=1,~d(x_1)=2,~n=18,\\
2[f_2(n-7)+1]+2[g(n-6)-1]+0+4g(n-7)+3g(n-8), &s=1,~d(x_1)=2,~n\neq18,\\
2[f_2(n-7)+1]+2g(n-6)+0+4g(n-7)+3f(n-9), &s=1,~d(x_1)\geq3,~|Y|=1,\\
2[f_2(n-7)+1]+2g(n-6)+0+4g(n-7)+3g(|Y|)f_2(n-|Y|-8), &s=1,~d(x_1)\geq3,~|Y|=4,5,\\
2[f_2(n-7)+1]+2g(n-6)+0+4g(n-7)+3\cdot\displaystyle\frac{15}{16}f_2(n-8), &s=1,~d(x_1)\geq3,~|Y|\geq6.\\
\end{cases}\\
<&g(n).
\end{align*}
}
The second inequality holds for the following reasons. For $s=2,~n = 15$, the second inequality holds because  $X$  and  $X - x$  cannot both have the structure shown in Fig.~\ref{fig2}. For $s=1,~d(x_1)=2$ and $n=18,$ neither $X$, $X-x$ nor $X-\{x,x_1\}$ attains the corresponding structure shown in Fig.~\ref{fig2}. Thus $\phi(X)<g(n-6)$, $\phi(X-x)<g(n-7)$ and $\phi(X-\{x,x_1\})<g(n-8)$. For $s=1,~d(x_1)=2$ and $n\neq18,$ the tree $X$ cannot attain the corresponding structure shown in Fig.~\ref{fig2}. For $s=1,~d(x_1)\geq3$ and $|Y|\geq4$, since there exists a component of order at least $5$ in $X - \{x, x_1\} \cup Y$, $\phi(X - \{x, x_1\} \cup Y)\leq f_2(n - |Y| - 8)$.

\textbf{Subcase 2.3.3.} $s=0.$

\begin{figure}[htbp]
\centering
\begin{minipage}[t]{0.105 \textwidth}
        \includegraphics[width=\linewidth]{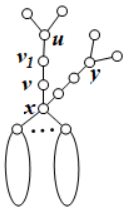}
        \caption{}
        \label{89}
    \end{minipage}
\hspace{0.2cm}
    \begin{minipage}[t]{0.118 \textwidth}
        \centering
       \includegraphics[width=\linewidth]{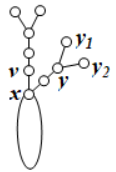}
        \caption{}
        \label{r=0s=0-1}
    \end{minipage}
\hspace{0.2cm}
     \begin{minipage}[t]{0.058 \textwidth}
        \includegraphics[width=\linewidth]{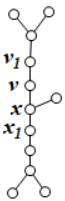}
        \caption{}
        \label{92}
    \end{minipage}
    \hspace{0.8cm}
    \begin{minipage}[t]{0.052 \textwidth}
        \includegraphics[width=\linewidth]{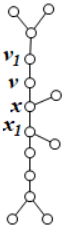}
        \caption{}
        \label{91}
    \end{minipage}
    \hspace{0.8cm}
    \begin{minipage}[t]{0.1 \textwidth}
        \includegraphics[width=\linewidth]{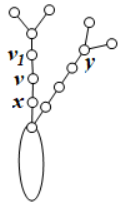}
        \caption{}
        \label{93}
    \end{minipage}
 \hspace{0.15cm}
     \begin{minipage}[t]{0.09 \textwidth}
        \includegraphics[width=\linewidth]{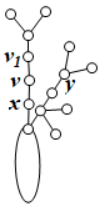}
        \caption{}
        \label{96}
    \end{minipage}
 \hspace{0.15cm}
\begin{minipage}[t]{0.094 \textwidth}
        \includegraphics[width=\linewidth]{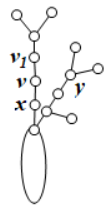}
        \caption{}
        \label{95}
    \end{minipage}
    \hspace{0.15cm}
\begin{minipage}[t]{0.093 \textwidth}
        \includegraphics[width=\linewidth]{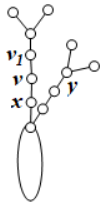}
        \caption{}
        \label{94}
    \end{minipage}
\end{figure}
If $d(x)\geq 3$, then each component $S$ of $\widetilde{T}-x$ disjoint from $P$ satisfies $\max\{d(u,x):u\in V(S)\}\in\{1,3,4\}.$ If there exists a component $S$ of $\widetilde{T}-x$ disjoint from $P$ satisfying $\max\{d(u,x):u\in V(S)\}=4,$ then only one graph remains to be verified; see Fig.~\ref{89}. By Theorems~\ref{S.ChengB.Wu} and Lemma \ref{zw4}, we have
{\begin{align*}
\phi(\widetilde{T})=&\phi(\widetilde{T},\bar{v}_1)+\phi(\widetilde{T},{v_1}^0)+\phi(\widetilde{T},{v_1}^1)\\
=&[\phi(\widetilde{T},\bar{v}_1\bar{u})+\phi(\widetilde{T},\bar{v}_1u\bar{y})+\phi(\widetilde{T},\bar{v}_1u{y}^0)+\phi(\widetilde{T},\bar{v}_1u{y}^1)]+\phi(\widetilde{T},{v_1}^0)\\
&+[\phi(\widetilde{T},{v_1}^1u\bar{y})+\phi(\widetilde{T},{v_1}^1u{y}^0)+\phi(\widetilde{T},{v_1}^1u{y}^1)+\phi(\widetilde{T},{v_1}^1v)]\\
\leq&
\begin{cases}
\{4f(n-12)+2g(n-7)+0+2[2g(n-8)+g(n-9)]\}+[f_2(n-6)+1]\\
+\{g(n-8)+0+[2g(n-9)+g(n-10)]+4g(n-11)\}, &d(x)=3,\\
\{4f(n-13)+2g(n-7)+0+2[2g(n-8)+g(n-9)]\}+[f_2(n-6)+1]\\
+\{g(n-8)+0+[2g(n-9)+g(n-10)]+4f(n-11)\}, &d(x)\geq4.\\
\end{cases}\\
<&g(n).
\end{align*}}
~~~If there exists a component $S$ of $\widetilde{T}-x$ disjoint from $P$ with $\max\{d(u,x):u\in V(S)\}=3$ (see Fig.~\ref{r=0s=0-1}), then neither $\widetilde{T}-\{y,y_1,y_2\}$ nor $\widetilde{T}-N[y]$ attains the corresponding structure shown in Fig.~\ref{fig2}. Therefore, we have $\phi(\widetilde{T})=\phi(\widetilde{T},\bar{y})+\phi(\widetilde{T},{y}^0)+\phi(\widetilde{T},{y}^1)\leq [g(n-3)-1]+0+\{2[g(n-4)-1]+4f(n-10)\}<g(n)$.

 Otherwise, the components of $\widetilde{T}-x$ disjoint from $P$ are isomorphic to $K_1$.

If the components of $\widetilde{T}-x$ that are disjoint from $P$ consist of two isolated vertices, then $\phi(\widetilde{T})=\phi(\widetilde{T},\bar{x})+\phi(\widetilde{T},{x}^0)+\phi(\widetilde{T},{x}^1\bar{x}_3)+\phi(\widetilde{T},{x}^1{x_3}).$ By Theorem~\ref{S.ChengB.Wu}, $\phi(\widetilde{T},{x}^1\bar{x}_3)\leq11f(n-9)$. On the basis of the preceding discussion, we have that $X - N[x]$ does not contain component isomorphic to $P_3$ and $X - N[x]\not\cong rK_{1,3}$ where $r\geq2$. If $X - N[x]\cong K_{1,3},$ then $n=13$. This contradicts $n \equiv 0,2 \pmod{3}$. Thus, $\phi(\widetilde{T},{x}^1\bar{x}_3)<11f(n-9)$. Therefore, $\phi(\widetilde{T})\leq 4g(n-8)+0+11f_2(n-9)+4f(n-10)<g(n)$.

If the components of $\widetilde{T}-x$ that are disjoint from $P$ consist of one isolated vertex, then $\phi(\widetilde{T})=\phi(\widetilde{T},\bar{x}\bar{v})+\phi(\widetilde{T},\bar{x}v)+\phi(\widetilde{T},{x}^0)+\phi(\widetilde{T},{x}^1\bar{x}_2)+\phi(\widetilde{T},{x}^1x_2)$. By Theorem~\ref{S.ChengB.Wu}, we have $\phi(\widetilde{T},{x}^1\bar{x}_2)\leq 7f(n-8)$ for $d(x_2)\geq3$ and $\phi(\widetilde{T},{x}^1x_2)\leq f(n-7-d(x_2))$. If $\phi(\widetilde{T},{x}^1\bar{x}_2)= 7f(n-8)$ for $d(x_2)\geq3$, then it follows from the condition of Case $2$ and the preceding discussion that $X -N[x]$ does not contain a component isomorphic to $P_3$ and that $X -N[x]\not\cong rK_{1,3}$ for any $r\geq2$. Thus, $\phi(\widetilde{T},{x}^1\bar{x}_2)< 7f(n-8)$ for $d(x_2)\geq3$. 
If $\phi(\widetilde{T},{x}^1x_2)= f(n-7-d(x_2)),$ then every component of $X-\{x,x_1\}-N[x_2]$ is isomorphic to either $P_3$ or $K_{1,3}$. If there exists a component of $X-\{x,x_1\}-N[x_2]$ isomorphic to $K_{1,3}$, then by the preceding discussion and $n \equiv 0,2 \pmod{3}$, we only need to consider the case $d(x_2)=3$ in which $x_2$ has exactly one leaf neighbor. See ~\ref{91}. Otherwise, it follows from the earlier discussion that $\widetilde{T}$ is as shown in Fig.~\ref{92}. By direct calculations, we obtain $\phi(\widetilde{T})<g(n).$  Thus, $\phi(\widetilde{T},{x}^1x_2)< f(n-7-d(x_2)).$ Therefore, by Theorems~\ref{S.ChengB.Wu} and~\ref{Wang,Zhang,Tu,Xiong} and Lemma \ref{zw4}, we have
{\small\begin{align*}
\phi(\widetilde{T})&=\phi(\widetilde{T},\bar{x}\bar{v})+\phi(\widetilde{T},\bar{x}v)+\phi(\widetilde{T},{x}^0)+\phi(\widetilde{T},{x}^1\bar{x}_2)+\phi(\widetilde{T},{x}^1x_2)\\
&\leq
\begin{cases}
[f_2(n-8)+1]+3g(n-7)+0+7g(n-8)+4f_2(n-9), &d(x_2)=2,\\
[f_2(n-8)+1]+3g(n-7)+0+7f_2(n-8)+4f_2(n-10), &d(x_2)\geq3.\\
\end{cases}\\
&<g(n).
\end{align*}}
~~~If $d(x)=2$, we first consider that $d(x_1)\geq3$. Then, each component $S$ of $\widetilde{T}-x_1$ disjoint from $P$ satisfies $\max\{d(u,x_1):u\in V(S)\}\in\{1,3,4,5\}.$ If there exists a component $S$ of $\widetilde{T}-x_1$ disjoint from $P$ satisfying $\max\{d(u,x_1):u\in V(S)\}=5,$ define the path $P_1=u_1u_2u_3u_4u_5x_1$, where $u_i\in V(S)$ for each $i\in [5]$. If either $d(u_4)\neq2$ or $d(u_5)\neq2$, then by analogy with the preceding proof, we have $\phi(\widetilde{T})<g(n)$. Then, only one graph remains to be verified; see Fig.~\ref{93}. If there exists a component $S$ of $\widetilde{T}-x_1$ disjoint from $P$ with $\max\{d(u,x_1):u\in V(S)\}=4$, then only three graphs remain to be verified; see Fig.~\ref{96},~\ref{95} and~\ref{94}. Therefore, by Theorems~\ref{S.ChengB.Wu} and Lemma \ref{zw4}, we have
{\small\begin{align*}
\phi(\widetilde{T})=&\phi(\widetilde{T},\bar{v}_1)+\phi(\widetilde{T},{v_1}^0)+\phi(\widetilde{T},{v_1}^1)\\
=&[\phi(\widetilde{T},\bar{v}\bar{v_1})+\phi(\widetilde{T},\bar{v}{v_1}^0)+\phi(\widetilde{T},\bar{v}{v_1}^1\bar{y})+\phi(\widetilde{T},\bar{v}{v_1}^1{y}^0)+\phi(\widetilde{T},\bar{v}{v_1}^1{y}^1)]+\phi(\widetilde{T},{v}^0)\\
&+[\phi(\widetilde{T},{v}^1\bar{x}\bar{y})+\phi(\widetilde{T},{v}^1\bar{x}{y}^0)+\phi(\widetilde{T},{v}^1\bar{x}{y}^1)+\phi(\widetilde{T},{v}^1x)]\\
\leq&
\begin{cases}
\{8f(n-14)+[f_2(n-6)+1]+g(n-8)+0+[2g(n-9)+g(n-10)]\}+2[f_2(n-7)+1]\\
+\{g(n-9)+0+[2g(n-10)+g(n-11)]+18f(n-13)\},\quad\quad \quad\text{if}~\widetilde{T} ~\text{is as shown in Fig. }~\ref{93},\\
\{8f(n-15)+[f_2(n-6)+1]+g(n-8)+0+[2g(n-9)+g(n-12)]\}+2[f_2(n-7)+1]\\
+\{g(n-9)+0+[2g(n-10)+g(n-13)]+33f(n-14)\}, \quad\quad \quad\text{if}~\widetilde{T} ~\text{is as shown in Fig. }~\ref{96},\\
\{8f(n-14)+[f_2(n-6)+1]+g(n-8)+0+[2g(n-9)+g(n-11)]\}+2[f_2(n-7)+1]\\
+\{g(n-9)+0+[2g(n-10)+g(n-12)]+18f(n-13)\}, \quad\quad \quad\text{if}~\widetilde{T} ~\text{is as shown in Fig. }~\ref{95},\\
\{8f(n-14)+[f_2(n-6)+1]+g(n-8)+0+[2g(n-9)+g(n-10)]\}+2[f_2(n-7)+1]\\
+\{g(n-9)+0+[2g(n-10)+g(n-11)]+18f(n-12)\}, \quad\quad \quad\text{if}~\widetilde{T} ~\text{is as shown in Fig. }~\ref{94},\\
\end{cases}\\
<&g(n).
\end{align*}}
\quad  If there exists a component $S$ of $\widetilde{T}-x_1$ disjoint from $P$ with $\max\{d(u,x_1):u\in V(S)\}=3$ (see Fig.~\ref{r=0s=0-3}), then $\phi(\widetilde{T})=\phi(\widetilde{T},\bar{y})+\phi(\widetilde{T},{y}^0)+\phi(\widetilde{T},{y}^1y_1)+\phi(\widetilde{T},{y}^1y_2)+\phi(\widetilde{T},{y}^1y_3).$  Based on the structure of $\widetilde{T}$ restricted by the preceding discussion and the choice of $P$, it follows that $X-\{x,x_1\}\cup N[y]$ must contain a component of order at least $5$. Thus $\phi(\widetilde{T},{y}^1y_1)<f(n-11)$ by Theorems~\ref{S.ChengB.Wu}. Therefore, $\phi(\widetilde{T})\leq g(n-3)+0+6f_2(n-11)+g(n-4)+g(n-4)<g(n)$ for $n\neq 18$. For $n = 18$, only one structure remains to be verified; see Fig.~\ref{97}, and all the other structures have already been shown to satisfy $\phi(\widetilde{T})<g(n)$ in the preceding cases. For the graph shown in Fig.~\ref{97}, direct computation yields $\phi(\widetilde{T})<g(n)$.

Otherwise, the components of $\widetilde{T}-x_1$ disjoint from $P$ are isomorphic to $K_1$. It follows from the preceding discussion that $X-N[x_1]$ must contain a component with a diameter of at least $3$. Then, $\phi(X-N[x_1])<f(n-d(x_1)-6)$ by Theorem~\ref{S.ChengB.Wu}. Therefore, by Theorems~\ref{S.ChengB.Wu} and~\ref{Wang,Zhang,Tu,Xiong}, we have
{\small\begin{align*}
\phi(\widetilde{T})&=\phi(\widetilde{T},\bar{x}_1\bar{x})+\phi(\widetilde{T},\bar{x}_1x^0)+\phi(\widetilde{T},\bar{x}_1x^1)+\phi(\widetilde{T},{x_1}^0)+\phi(\widetilde{T},{x_1}^1)\\
&\leq
\begin{cases}
g(n-9)+2g(n-9)+3g(n-9)+0+[12f_2(n-10)+4f(n-11)],~~~ x_1~\text{has two leaf neighbors} ,\\
[f_2(n-9)+1]+2g(n-8)+3g(n-8)+0+[8f_2(n-9)+4f(n-10)], ~~x_1~\text{has one leaf neighbor}.\\
\end{cases}\\
&<g(n).
\end{align*}}
~~~~If $d(x_1)=2$, let $N(x_1)=\{x,z_1\}$. We first consider that $d(z_1)\geq3$. Then, each component $S$ of $\widetilde{T}-z_1$ disjoint from $P$ satisfies $\max\{d(u,z_1):u\in V(S)\}\in\{1,3,4,5,6\}.$ If there exists a component $S$ of $\widetilde{T}-z_1$ disjoint from $P$ satisfying $\max\{d(u,z_1):u\in V(S)\}\in\{5,6\},$ the path $P_1=u_1u_2\ldots u_kz_1$ is defined, where $k=5,6$ and $u_i\in V(S)$ for each $i\in [k]$. If there exists $i\in[4,k]$ such that $d(u_i)\neq2$, then by analogy with the preceding proof we have $\phi(\widetilde{T})<g(n)$. Only five graphs remain to be verified; see Fig.~\ref{98},~\ref{99},~\ref{100},~\ref{101} and~\ref{102}. Therefore, by Theorems~\ref{S.ChengB.Wu} and~\ref{Wang,Zhang,Tu,Xiong} and Lemma \ref{zw4}, we have
{\small\begin{align*}
\phi(\widetilde{T})&=\phi(\widetilde{T},\bar{x}\bar{v})+\phi(\widetilde{T},\bar{x}{v}^0)+\phi(\widetilde{T},\bar{x}{v}^1\bar{y})+\phi(\widetilde{T},\bar{x}{v}^1{y}^0)+\phi(\widetilde{T},\bar{x}{v}^1{y}^1)+\phi(\widetilde{T},{x}^0)+\phi(\widetilde{T},{x}^1v)+\phi(\widetilde{T},{x}^1x_1)\\
&\leq
\begin{cases}
6f(n-16)+2[f_2(n-7)+1]+g(n-9)+0+[2g(n-10)+g(n-11)]\\
+2[f_2(n-8)+1]+3g(n-7)+40f_2(n-15),\qquad\qquad \quad\text{if}~\widetilde{T} ~\text{is as shown in Fig.}~\ref{98},\\
4f(n-15)+2[f_2(n-7)+1]+g(n-9)+0+[2g(n-10)+g(n-11)]\\
+2[f_2(n-8)+1]+3g(n-7)+24f_2(n-14), \qquad\qquad \quad\text{if}~\widetilde{T} ~\text{is as shown in Fig.}~\ref{99},\\
4f(n-16)+2[f_2(n-7)+1]+g(n-9)+0+[2g(n-10)+g(n-11)]\\
+2[f_2(n-8)+1]+3g(n-7)+44f_2(n-15), \qquad\qquad \quad\text{if}~\widetilde{T} ~\text{is as shown in Fig.}~\ref{100},\\
4f(n-15)+2[f_2(n-7)+1]+g(n-9)+0+[2g(n-10)+g(n-11)]\\
+2[f_2(n-8)+1]+3g(n-7)+24f_2(n-14), \qquad\qquad \quad\text{if}~\widetilde{T} ~\text{is as shown in Fig.}~\ref{101},\\
4f(n-14)+2[f_2(n-7)+1]+g(n-9)+0+[2g(n-10)+g(n-11)]\\
+2[f_2(n-8)+1]+3g(n-7)+16f_2(n-14), \qquad\qquad \quad\text{if}~\widetilde{T} ~\text{is as shown in Fig.}~\ref{102}.\\
\end{cases}\\
&<g(n).
\end{align*}}
\begin{figure}[htbp]
\begin{minipage}[t]{0.095 \textwidth}
        \includegraphics[width=\linewidth]{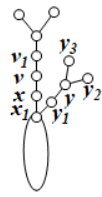}
        \caption{}
        \label{r=0s=0-3}
    \end{minipage}
    \hspace{0.3cm}
  \begin{minipage}[t]{0.07 \textwidth}
        \includegraphics[width=\linewidth]{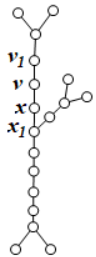}
        \caption{}
        \label{97}
    \end{minipage}
 \hspace{0.3cm}
     \begin{minipage}[t]{0.1 \textwidth}
        \includegraphics[width=\linewidth]{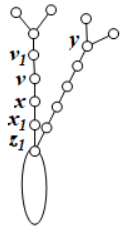}
        \caption{}
        \label{98}
    \end{minipage}
    \hspace{0.3cm}
     \begin{minipage}[t]{0.092 \textwidth}
        \includegraphics[width=\linewidth]{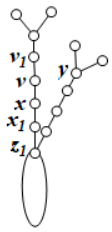}
        \caption{}
        \label{99}
    \end{minipage}
    \hspace{0.3cm}
     \begin{minipage}[t]{0.09 \textwidth}
        \includegraphics[width=\linewidth]{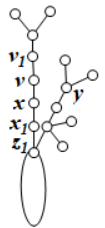}
        \caption{}
        \label{100}
    \end{minipage}
    \hspace{0.3cm}
     \begin{minipage}[t]{0.087 \textwidth}
        \includegraphics[width=\linewidth]{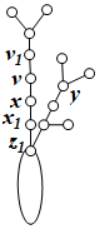}
        \caption{}
        \label{101}
    \end{minipage}
    \hspace{0.3cm}
     \begin{minipage}[t]{0.088 \textwidth}
        \includegraphics[width=\linewidth]{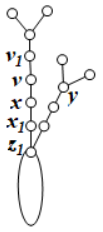}
        \caption{}
        \label{102}
    \end{minipage}
 \hspace{0.3cm}
 \begin{minipage}[t]{0.088 \textwidth}
        \includegraphics[width=\linewidth]{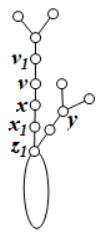}
        \caption{}
        \label{r=0s=0-4}
    \end{minipage}
\end{figure}

 If there exists a component $S$ of $\widetilde{T}-z_1$ disjoint from $P$ with $\max\{d(u,z_1):u\in V(S)\}=3$ (see Fig.~\ref{r=0s=0-4}), then let $Y$ be a component of $X-\{x,x_1,z_1\}\cup N[y]$ with the minimum order. Thus, by Theorems~\ref{S.ChengB.Wu} and~\ref{Wang,Zhang,Tu,Xiong} and Lemma \ref{1}, we have
{\begin{align*}
\phi(\widetilde{T})&=\phi(\widetilde{T},\bar{y})+\phi(\widetilde{T},y^0)+\phi(\widetilde{T},y^1)\\
&\leq
\begin{cases}
g(n-3)+2g(n-4)+10g(n-12), &d(x_2)=3,\\
g(n-3)+2g(n-4)+10f(n-13), &d(x_1)\geq4,~|Y|=1,\\
g(n-3)+2g(n-4)+10g(|Y|)f_2(n-|Y|-12), &d(x_1)\geq4,~|Y|=4,5,\\
g(n-3)+2g(n-4)+10\cdot\frac{15}{16}f_2(n-12), &d(x_1)\geq4,~|Y|\geq6.\\
\end{cases}\\
&<g(n).
\end{align*}}
~~~~Otherwise, the components of $\widetilde{T}-z_1$ disjoint from $P$ are isomorphic to $K_1$. Then, by Theorems~\ref{S.ChengB.Wu} and~\ref{Wang,Zhang,Tu,Xiong} and Lemma \ref{zw4}, we have
{\small\begin{align*}
&\phi(\widetilde{T})\\
=&\phi(\widetilde{T},\bar{z}_1\bar{x}_1)+\phi(\widetilde{T},\bar{z}_1{x_1}^0)+\phi(\widetilde{T},\bar{z}_1{x_1}^1)+\phi(\widetilde{T},{z_1}^0)+\phi(\widetilde{T},{z_1}^1)\\
\leq&
\begin{cases}
3g(n-10)+3g(n-10)+4g(n-10)+0+[16f(n-11)+6f(n-12)], ~~ z_1 ~\text{has two leaf neighbors},\\
3[f_2(n-10)+1]+3g(n-10)+4g(n-10)+0+[10f(n-10)+6f(n-11)], ~~ z_1 ~\text{has one leaf neighbor}.\\
\end{cases}\\
<&g(n).
\end{align*}}
~~~If $d(z_1)=2$, let $N(z_1)=\{x_1,z_2\}$. We first consider that $d(z_2)\geq3$. Then, each component $S$ of $\widetilde{T}-z_2$ disjoint from $P$ satisfies $\max\{d(u,z_2):u\in V(S)\}\in\{1,3,4,5,6,7\}.$ If there exists a component $S$ of $\widetilde{T}-z_2$ disjoint from $P$ satisfying $\max\{d(u,z_2):u\in V(S)\}\in\{5,6,7\},$ the path $P_1=u_1u_2\ldots u_kz_2$ is defined, where $k=5,6,7$ and $u_i\in V(S)$ for each $i\in [k]$. If there exists $i\in[4,k]$ such that $d(u_i)\neq2$, then by analogy with the preceding proof, we have $\phi(\widetilde{T})<g(n)$. Then, only three graphs remain to be verified; see Fig.~\ref{103},~\ref{104} and~\ref{105}. If there exists a component $S$ of $\widetilde{T}-z_2$ disjoint from $P$ with $\max\{d(u,z_2):u\in V(S)\}=4$, then only three graphs remain to be verified; see Fig.~\ref{107},~\ref{108} and~\ref{106}.
\begin{figure}[htbp]
 \centering
\begin{minipage}[t]{0.09 \textwidth}
        \includegraphics[width=\linewidth]{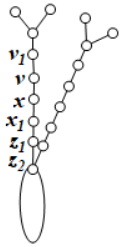}
        \caption{}
        \label{103}
    \end{minipage}
\hspace{0.3cm}
     \begin{minipage}[t]{0.086 \textwidth}
        \includegraphics[width=\linewidth]{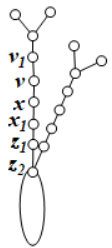}
        \caption{}
        \label{104}
    \end{minipage}
\hspace{0.3cm}
    \begin{minipage}[t]{0.079 \textwidth}
        \includegraphics[width=\linewidth]{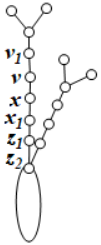}
        \caption{}
        \label{105}
    \end{minipage}
\hspace{0.4cm}
    \begin{minipage}[t]{0.076\textwidth}
        \includegraphics[width=\linewidth]{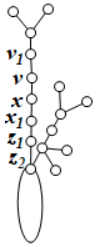}
        \caption{}
        \label{107}
    \end{minipage}
\hspace{0.4cm}
    \begin{minipage}[t]{0.077 \textwidth}
        \includegraphics[width=\linewidth]{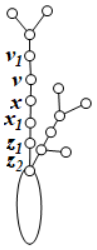}
        \caption{}
        \label{108}
    \end{minipage}
\hspace{0.4cm}
     \begin{minipage}[t]{0.078 \textwidth}
        \includegraphics[width=\linewidth]{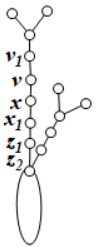}
        \caption{}
        \label{106}
    \end{minipage}
\hspace{0.4cm}
 \begin{minipage}[t]{0.085\textwidth}
        \includegraphics[width=\linewidth]{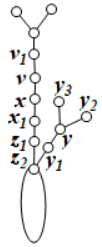}
        \caption{}
        \label{r=0s=0-5}
    \end{minipage}
    \hspace{0.4cm}
     \begin{minipage}[t]{0.048 \textwidth}
        \includegraphics[width=\linewidth]{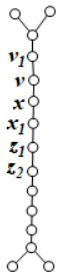}
        \caption{}
        \label{111}
    \end{minipage}
\end{figure}
 Therefore, by Theorems~\ref{S.ChengB.Wu} and~\ref{Wang,Zhang,Tu,Xiong} and Lemma \ref{zw4}, we have
{\small\begin{align*}
\phi(\widetilde{T})&=\phi(\widetilde{T},\bar{x}_1)+\phi(\widetilde{T},{x_1}^0)+\phi(\widetilde{T},{x_1}^1)\\
&=[\phi(\widetilde{T},\bar{x}_1\bar{x})+\phi(\widetilde{T},\bar{x}_1{x}^0)+\phi(\widetilde{T},\bar{x}_1{x}^1)]+\phi(\widetilde{T},{x_1}^0)+[\phi(\widetilde{T},{x_1}^1x)+\phi(\widetilde{T},{x_1}^1x_2)]\\
&\leq
\begin{cases}
\{10f(n-18)+2[f_2(n-8)+1]+3g(n-7)\}+3[f_2(n-9)+1]\\
+[4g(n-8)+52f(n-17)],\qquad\qquad \quad\qquad\qquad\text{if}~\widetilde{T} ~\text{is as shown in Fig.}~\ref{103},\\
\{6f(n-17)+2[f_2(n-8)+1]+3g(n-7)\}+3[f_2(n-9)+1]\\
+[4g(n-8)+40f(n-16)], \qquad\qquad\qquad\qquad \quad\text{if}~\widetilde{T} ~\text{is as shown in Fig.}~\ref{104},\\
\{4f(n-16)+2[f_2(n-8)+1]+3g(n-7)\}+3[f_2(n-9)+1]\\
+[4g(n-8)+24f(n-15)], \qquad\qquad\qquad\qquad \quad\text{if}~\widetilde{T} ~\text{is as shown in Fig.}~\ref{105},\\
\{4f(n-17)+2[f_2(n-8)+1]+3g(n-7)\}+3[f_2(n-9)+1]\\
+[4g(n-8)+44f(n-16)], \qquad\qquad\qquad\qquad \quad\text{if}~\widetilde{T} ~\text{is as shown in Fig.}~\ref{107},\\
\{4f(n-16)+2[f_2(n-8)+1]+3g(n-7)\}+3[f_2(n-9)+1]\\
+[4g(n-8)+24f(n-15)],\qquad\qquad \qquad\qquad \quad\text{if}~\widetilde{T} ~\text{is as shown in Fig.}~\ref{108},\\
\{4f(n-15)+2[f_2(n-8)+1]+3g(n-7)\}+3[f_2(n-9)+1]\\
+[4g(n-8)+16f(n-14)],\qquad\qquad\qquad\qquad \quad\text{if}~\widetilde{T} ~\text{is as shown in Fig.}~\ref{106}.\\
\end{cases}\\
&<g(n).
\end{align*}}
\quad  If there exists a component $S$ of $\widetilde{T}-z_2$ disjoint from $P$ with $\max\{d(u,z_2):u\in V(S)\}=3$ (see Fig.~\ref{r=0s=0-5}), then $n\geq21$. By the induction hypothesis, we have $\phi(\widetilde{T})=\phi(\widetilde{T},\bar{y})+\phi(\widetilde{T},{y}^0)+\phi(\widetilde{T},{y}^1)=g(n-3)+0+[2g(n-4)+13f(n-13)]<g(n)$.

Otherwise, the components of $\widetilde{T}-z_2$ disjoint from $P$ are isomorphic to $K_1$.
Then, by Theorems~\ref{S.ChengB.Wu} and~\ref{Wang,Zhang,Tu,Xiong} and Lemma \ref{zw4}, we have
{\small\begin{align*}
\phi(\widetilde{T})=&
\begin{cases}
\phi(\widetilde{T},\bar{x}_3)+\phi(\widetilde{T},{x_3}^0)+\phi(\widetilde{T},{x_3}^1),  &x_3 ~\text{has two leaf neighbors},\\
\phi(\widetilde{T},\bar{x}_3\bar{x}_2)+\phi(\widetilde{T},\bar{x}_3{x}_2)+\phi(\widetilde{T},{x_3}^0)+\phi(\widetilde{T},{x_3}^1), & x_3 ~\text{has one leaf neighbor}.\\
\end{cases}\\
\leq&
\begin{cases}
13g(n-11)+0+[26f(n-12)+10f(n-13)],  &x_3 ~\text{has two leaf neighbors},\\
4[f_2(n-11)+1]+9g(n-10)+0+[16f_2(n-11)+10f(n-12)], & x_3 ~\text{has one leaf neighbor}.\\
\end{cases}\\
<&g(n).
\end{align*}}
\quad If $d(z_2)=2$, we have  $\phi(\widetilde{T})
=\phi(\widetilde{T},\bar{z}_1\bar{x}_1)+\phi(\widetilde{T},\bar{z}_1{x_1}^0)+\phi(\widetilde{T},\bar{z}_1{x_1}^1)+\phi(\widetilde{T},{z_1}^0)+\phi(\widetilde{T},{z_1}^1x_1)+\phi(\widetilde{T},{z_1}^1z_2).$
Let $N(z_2)=\{z_1,z_3\}$. For $d(z_3)\geq3$, if $X-\{x,x_1,z_1,z_2\}\cong P_3$, then by direct calculations, we have $\phi(\widetilde{T})<g(n).$ If $X-\{x,x_1,z_1,z_2\}\not\cong P_3$, then by the preceding discussion, the tree $X-\{x,x_1,z_1,z_2\}$ has a diameter of at least $8$. 
Therefore, $\phi(\widetilde{T},{z_1}^1z_2)<6f(n-10)$ by Theorem~\ref{S.ChengB.Wu}. Thus, by Theorems~\ref{S.ChengB.Wu} and~\ref{Wang,Zhang,Tu,Xiong} and Lemma \ref{zw4}, we have $\phi(\widetilde{T})\leq3f_2(n-12)+3[f_2(n-9)+1]+4g(n-8)+5[f_2(n-10)+1]+4g(n-9)+6f_2(n-10)<g(n).$

 For $d(z_3)=2$, if $|X-\{x,x_1,z_1\}-N[z_3]|=2$ or $3$, then by direct calculations, we have $\phi(\widetilde{T})<g(n).$ If $|X-\{x,x_1,z_1\}-N[z_3]|\geq4,$ by Theorem~\ref{S.ChengB.Wu}, we have $\phi(\widetilde{T},\bar{z}_1\bar{x}_1)\leq3f(n-11).$ If $\phi(\widetilde{T},\bar{z}_1\bar{x}_1)=3f(n-11),$ then $Y-N[z_3]\cong P_3$ or $K_{1,3}$. By direct calculations, we have $\phi(\widetilde{T})<g(n).$ If $\phi(\widetilde{T},\bar{z}_1{x_1}^1)=4g(n-8),$ then $\widetilde{T}$ is as shown in Fig. \ref{111}. By direct calculations, we have $\phi(\widetilde{T})<g(n).$ 
Therefore, by Theorems~\ref{S.ChengB.Wu} and~\ref{Wang,Zhang,Tu,Xiong} and Lemma \ref{zw4}, we have
$\phi(\widetilde{T})\leq3f_2(n-11)+3[f_2(n-9)+1]+4[g(n-8)-1]+5[f_2(n-10)+1]+4g(n-9)+6g(n-10)<g(n).$

Thus, we complete the proof of Theorem~\ref{main}.\hfill$\square$

\label{app}

\end{document}